\documentclass[10pt,reqno]{amsart}
\usepackage{latexsym}
\usepackage{amssymb}
\usepackage{amsmath}
\usepackage{bm}
\usepackage[english]{babel}
\usepackage[dvips]{graphicx}
\usepackage{texdraw}
\usepackage{graphics}
\font\tenmath=msbm10 \font\sevenmath=msbm7 \font\fivemath=msbm5
\newfam\mathfam \textfont\mathfam=\tenmath
\scriptfont\mathfam=\sevenmath \scriptscriptfont\mathfam=\fivemath

\def \\ { \cr }

\newcommand{\RR}{{\mathbb R}}
\newcommand{\CC}{{\mathbb C}}
\newcommand{\NN}{{\mathbb N}}
\newcommand{\ZZ}{\mathbb Z}

\providecommand{\abs}[1]{\left\lvert {#1} \right\rvert}

\def\H{{\mathcal H}}

\def\P{{\mathcal P}}

\def\T{{\mathcal T}}

\newcommand\la{\langle}
\newcommand\ra{\rangle}

\newcommand{\rb}{\mathbf{b}}
\newcommand{\rB}{\mathbf{B}}

\newcommand{\set}[1]{\left\{{#1}\right\}}
\newcommand{\implica}{\Longrightarrow}

\newcommand{\talque}{;}
\newcommand{\y}{ \; \textrm{and} \; }
\newcommand{\ds}{\displaystyle}

\newcommand{\econd}[2]{\mathbb{E}({#1}|{#2})}
\newcommand{\torref}[1]{\tau_{#1}}

\newcommand{\sufijofb}[1]{\overline{s}_{#1}}                    
\newcommand{\sufijofbb}[2]{\overline{s}_{{#1},{#2}}}            
\newcommand{\sufijocb}[3]{\overline{S}_{{#1}}({#2},{#3})}       
\newcommand{\sufijocbb}[4]{\overline{S}_{{#1},{#2}}({#3},{#4})} 
\newcommand{\suma}[4]{\sigma_{{#1},{#2}}({#3},{#4})}
\newcommand{\sumacoef}[5]{{\bm{\sigma}}_{{#1},{#2}}^{({#5})}({#3},{#4})}
\newcommand{\indice}[5]{{#1}_{{#2},{#3}}({#4},{#5})}
\newcommand{\matrizp}[4]{P_{{#1},{#2}}({#3},{#4})}               
\newcommand{\ld}{\lambda}

\numberwithin{equation}{section}
\newtheorem{theo}{Theorem}

\newtheorem{coro}[theo]{Corollary}
\newtheorem{lemm}[theo]{Lemma}

\theoremstyle{remark}

\begin{document}

\title
{Eigenvalues of Toeplitz minimal systems of finite topological rank}

\author{Fabien Durand}
\address{Laboratoire Ami\'enois
de Math\'ematiques Fondamentales et Appliqu\'ees, CNRS-UMR 7352, Universit\'{e} de Picardie Jules Verne, 33 rue Saint Leu, 80039   Amiens cedex 1,
France.} \email{fabien.durand@u-picardie.fr}

\author{Alexander Frank}
\address{Departamento de Ingenier\'{\i}a
Matem\'atica, Universidad de Chile, Avenida Blanco Encalada 2120, Santiago, Chile.}
\email{afrank@dim.uchile.cl}

\author{Alejandro Maass}
\address{Departamento de Ingenier\'{\i}a
Matem\'atica and Centro de Modelamiento Ma\-te\-m\'a\-ti\-co, CNRS-UMI 2807, Universidad de Chile, Avenida Blanco Encalada 2120, Santiago,
Chile.}\email{amaass@dim.uchile.cl}


\thanks{This research was partially supported by grants Basal-CMM \& Fondap 15090007, Proyecto Anillo ACT1103, ANR Subtile and the cooperation project Mathamsud DYSTIL. The third author thanks Chaire B\'ezout at University of Paris Est and U. Picardie Jules Verne where this research was finished. \\
We are very grateful to the anonymous referee who helped significantly to improve the presentation of this article.}

\begin{abstract}
In this article we characterize measure theoretical eigenvalues of Toeplitz Bratteli-Vershik minimal systems of finite topological rank which are not associated to a continuous eigenfunction.
Several examples are provided to illustrate the different situations that can occur.
\end{abstract}

\date{March 16, 2014}

\maketitle

\markboth{Fabien Durand, Alexander Frank, Alejandro Maass}{Eigenvalues of finite rank Toeplitz systems}

\parindent=0cm 

\section{Introduction}

Seminal results by M. Dekking \cite{Dekking} and B. Host \cite{Host} state that eigenvalues of primitive substitution dynamical systems are always associated to continuous eigenfunctions. Thus the topological and measure theoretical Kronecker factors coincide. It is natural to ask whether this phenomenon is still true for other classes of minimal Cantor systems. Most of the answers we have are negative. 

Substitution dynamical systems correspond to expansive minimal Cantor systems having a periodic or stationary Bratteli-Vershik representation \cite{dhs}. A natural class to explore extending the former one are linearly recurrent minimal Cantor systems, which correspond to those systems having a Bratteli-Vershik representation with a bounded number of incidence matrices. In \cite{lr} and \cite{necesariasuficiente} necessary and sufficient conditions based only on the combinatorial structure of the Bratteli diagrams are given for this class of systems, allowing  to differentiate continuous and measure theoretical but non continuous eigenvalues. The more general class of topological finite rank minimal Cantor systems is explored in \cite{rangofinito}, providing new examples and conditions to differentiate the topological and measure theoretical Kronecker factors.

It is known that any countable subgroup of the torus $\mathbb{S}^1=\{z \in \CC \ ; |z|=1 \}$ containing infinitely many rationals can be the set of eigenvalues of a Toeplitz system \cite{iwanik,toeplitzpreset}. Nevertheless, in the class of finite rank systems, Toeplitz systems exhibit a completely different behavior. Indeed, if a Toeplitz system is linearly recurrent then all its eigenvalues are associated to continuous eigenfunctions and if it has finite topological rank  just a few extra non continuous eigenvalues can appear and they are rational \cite{rangofinito}. 
So the assumption of finite topological rank restricts the possibilities of non continuous eigenvalues to some particular ones. The purpose of this work is to study the nature of these particular non continuous eigenvalues of finite rank Toeplitz systems. 
\medskip

Our main result (Theorem \ref{theo:principal}) states a necessary and sufficient condition for $\lambda=\exp{(2i\pi a/b)}$,
where $a,b$ are integers with $(a,b)=1$, to be a non continuous eigenvalue of a finite topological rank Toeplitz system. This condition shows that non continuous eigenvalues are very rare and impose particular local orders to the associated Bratteli-Vershik representations. In addition, even if this condition looks abstract, it is easily computable and allows to produce concrete examples, showing particular behaviors of the group of eigenvalues in relation to the set of ergodic measures.  
\medskip

The article is organized as follows. Section 2 contains the main definitions concerning eigenvalues of dynamical systems and Bratteli-Vershik representations, in particular the concept of Toeplitz minimal Cantor system of finite topological rank. In Section 3 we give the main result of the article and its corollaries. In particular, we exhibit a relation between the number of ergodic measures and the number of non continuous eigenvalues in the class of Toeplitz minimal Cantor systems of finite topological rank. Main technical lemmas used in the proofs are given in Section 4 and the proofs of the main result and its corollaries in Section 5. Finally, in Section 6 we provide several examples to illustrate the main result, its consequences and the fact that our condition is computable.

\section{Basic definitions}
\subsection{Dynamical systems and eigenvalues}
A \emph{topological dynamical system}, or just dynamical system, is a compact Hausdorff space $X$ together with a homeomorphism 
$T:X\rightarrow X$. We use the notation $\left( X,T\right)$. If $X$ is a Cantor set 
(i.e., $X$ has a countable basis of closed and open sets and it has no isolated points) we say that the system is Cantor. A dynamical system is \emph{minimal} if all orbits are dense in $X$,
or equivalently the only non empty closed invariant set is $X$.

A complex number $\lambda$ is a {\it continuous eigenvalue} of $(X,T)$ if there exists a continuous function $f : X\to \CC$, $f\not = 0$, such
that $f\circ T = \lambda f$; $f$ is called a {\it continuous eigenfunction} (associated to $\lambda$). Let $\mu$ be a $T$-invariant probability measure, i.e., $T\mu = \mu$, defined on the Borel $\sigma$-algebra of $X$. A complex number $\lambda$ is an {\it eigenvalue} of the
dynamical system $(X,T)$ with respect to $\mu$ if there exists $f\in L^2 (X,\mu)$, $f\not = 0$, such that $f\circ T = \lambda f$; $f$ is
called an {\it eigenfunction} (associated to $\lambda$). If $\mu$ is ergodic, then every eigenvalue has modulus 1 and every eigenfunction has a constant modulus $\mu$-almost surely. Of course, continuous eigenvalues are eigenvalues.

\subsection{Bratteli-Vershik representations}
Let $(X,T)$ be a minimal Cantor system. It can be represented by an ordered Bratteli diagram together with the Vershik transformation acting on it. For details on this theory see \cite{hps} or \cite{review}. This couple is called a Bratteli-Vershik representation of the system. We give a brief outline of this construction emphasizing the notations in this paper.

\subsubsection{Bratteli diagrams}
A Bratteli diagram is an infinite graph $\left( V,E\right)$ which consists of a vertex set $V$ 
and an edge set $E$, both of
which are divided into levels $V=V_{0}\cup V_{1}\cup \ldots$, $E=E_{1}\cup E_{2}\cup \ldots$ 
and all levels are pairwise disjoint. 
The set $V_{0}$ is a singleton $\{v_{0}\}$ and for all $n\geq 1$ edges in $E_{n}$ join vertices in $V_{n-1}$ to
vertices in $V_{n}$. It is also required that every vertex in $V_{n}$ is the ``end-point'' of some edge in $E_{n}$ for $n\geq 1$ and an ``initial-point'' of some edge in $E_{n+1}$ for $n\geq 0$. We set $\#V_{n}=d_{n}$ for all $n\geq 1$. 
\smallskip

Fix $n\geq 1$. We call \emph{level} $n$ of the diagram to the subgraph consisting of the vertices in $V_{n-1}\cup V_{n}$ 
and the edges $E_{n}$ between these vertices. Level $1$ is called the \emph{hat} of the Bratteli diagram. 
We describe the edge set $E_n$ using a $V_{n-1}\times V_{n}$ incidence matrix $M_{n}$ for which 
its $(t_{1},t_{2})$ entry is the number of edges in $E_{n}$ joining vertex $t_{1}\in V_{n-1}$ 
with vertex $t_{2} \in V_{n}$. We also set $P_{n}=M_{2}\cdots M_{n}$ with the convention that $P_{1}=I$, where $I$ denotes the identity matrix. The number of paths joining $v_{0} \in V_{0}$ and a vertex $t\in V_{n}$ is given by coordinate $t$ of the \emph{height row vector} $h_{n}=\left(h_n(t) ; t \in V_{n} \right) \in \NN ^{d_{n}}$. Notice that $h_{1}=M_{1}$ and $h_{n}=h_{1}P_{n}$. 
\smallskip

We also consider several levels at the same time. 
For integers $0 \leq m< n$ we denote by $E_{m,n}$ the set of all paths in the graph joining vertices of  
$V_{m}$ with vertices of $V_{n}$. We define matrices $P_{m,n}=M_{m+1} \cdots M_{n}$ with the convention that $P_{n,n}=I$
for $1\leq m \leq n$. Clearly, coordinate $\matrizp{m}{n}{t_{1}}{t_{2}}$ of matrix $P_{m,n}$ is the number of paths in $E_{m,n}$ from vertex $t_{1} \in V_{m}$ to vertex $t_{2} \in V_{n}$. 
It can be verified that $h_{n}=h_{m}P_{m,n}$. 
\smallskip

We need to notice that the incidence matrices defined above correspond to the transpose of the matrices defined at the classical reference in this theory \cite{hps}. This choice is done to simplify the understanding and reading of the article.  

\subsubsection{Ordered Bratteli diagrams and Bratteli-Vershik representations}
An \emph{ordered} Bratteli diagram is a triple \( B=\left( V,E,\preceq \right) \), where \( \left( V,E\right)  \) is a Bratteli diagram and \( \preceq  \) is a partial ordering on \( E \) such that: edges
\( e \) and \( e' \) are comparable if and only if they have the same end-point.
This partial ordering naturally defines maximal and minimal edges and paths. Also, 
the partial ordering of $E$ induces another one on paths of $E_{m,n}$, where $0 \leq m < n$:  
$\left(e_{m+1},\ldots,e_{n}\right) \preceq \left(f_{m+1},\ldots ,f_{n}\right)$ if and only if 
there is $m+1\leq i\leq n$ such that $e_{j}=f_{j}$ for $i<j\leq n$ and $e_{i}\preceq f_{i}$.

Given a strictly increasing sequence of integers 
$\left(n_{k}\right)_{k\geq 0}$ with $n_{0}=0$ one defines the \emph{contraction} or \emph{telescoping} of
$B=\left(V,E,\preceq \right)$ with respect to $\left(n_{k} \right)_{k\geq 0}$ as 
$$\left(\left(V_{n_{k}}\right)_{k\geq
0},\left( E_{n_{k},n_{k+1}}\right)_{k\geq 0},\preceq \right), $$ where $\preceq$ is the order induced in each set of edges 
$E_{n_{k},n_{k+1}}$. The converse operation is called {\it microscoping} (see \cite{hps} for more details).
\smallskip

Given an ordered Bratteli diagram \( B=\left( V,E,\preceq \right) \) one defines \( X_{B} \) as the set of infinite paths \( \left(
x_{1},x_{2},\ldots \right)  \) starting in \( v_{0} \) such that for all \( n\geq 1 \) the end-point of \( x_{n}\in E_{n} \) is the
initial-point of \( x_{n+1}\in E_{n+1} \). We topologize \( X_{B} \) by postulating a basis of open sets, namely the family of \emph{cylinder sets}
$$
\left[ e_{1},e_{2},\ldots ,e_{n}\right]
=\left\{ \left( x_{1},x_{2},\ldots \right) \in X_{B} \textrm{ } ; \textrm{ }
x_{i}=e_{i},\textrm{ for }1\leq i\leq n\textrm{ }
\right\} .$$

Each \( \left[ e_{1},e_{2},\ldots ,e_{n}\right]  \) is also closed, as is
easily seen, and so \( X_{B} \) is a compact, totally disconnected metrizable space.

When there is a unique \( \left( x_{1},x_{2},\ldots \right) \in X_{B} \) such that \( x_{n} \) is (locally) maximal for any \( n\geq 1 \) and a unique
\( \left( y_{1},y_{2},\ldots \right) \in X_{B} \) such that \( y_{n} \) is (locally) minimal for any \(n \geq 1 \), one says that \( B=\left(
V,E,\preceq \right)  \) is a \emph{properly ordered} Bratteli diagram. Call these particular points \( x_{\mathrm{max}} \) and \(
x_{\mathrm{min}} \) respectively. In this case one defines the dynamic \( V_{B} \) over \( X_{B} \) called the \emph{Vershik map}. Let \( x=\left( x_{1},x_{2},\ldots \right) \in X_{B}\setminus \left\{ x_{\mathrm{max}}\right\}  \) and let \(
n\geq 1 \) be the smallest integer so that \( x_{n} \) is not a maximal edge. Let \( y_{n} \) be the successor of \( x_{n} \) for the local order and \( \left(
y_{1},\ldots ,y_{n-1}\right)  \) be the unique minimal path in \( E_{0,n-1} \) connecting \( v_{0} \) with the initial vertex of \( y_{n} \). One
sets \( V_{B}\left( x\right) =\left( y_{1},\ldots ,y_{n-1},y_{n},x_{n+1},\ldots \right)  \) and \( V_{B}\left( x_{\mathrm{max}}\right)
=x_{\mathrm{min}} \).

The dynamical system \( \left( X_{B},V_{B}\right)  \) is minimal. It is called the \emph{Bratteli-Vershik system} generated by \( B=\left( V,E,\preceq \right)
\). The dynamical system induced by any telescoping of \( B \) is topologically conjugate to \( \left( X_{B},V_{B}\right)  \). In \cite{hps} it
is proved that any minimal Cantor system \( \left( X,T\right)  \) is topologically conjugate to a Bratteli-Vershik system \( \left(
X_{B},V_{B}\right)  \). One says that \( \left( X_{B},V_{B}\right) \) is a \emph{Bratteli-Vershik representation} of \( \left( X,T\right)  \).
In what follows we identify $(X,T)$ with any of its Bratteli-Vershik representations.

\subsubsection{Minimal Cantor systems of finite topological rank}

A minimal Cantor system is of finite (topological)  rank if it admits a  Bratteli-Vershik representation such that the number of vertices per level is uniformly bounded by some integer $d$. The minimum possible value of $d$ is called the \emph{topological rank} of the system. We observe that topological and measure theoretical finite rank  notions are completely different. For instance, systems of topological rank one correspond to odometers, whereas in the measure theoretical sense there are rank one systems that are expansive as classical Chacon's example. 

To have a better understanding of the dynamics of a minimal Cantor system, and in particular 
to understand its group of eigenvalues, one needs to work with a ``good'' Bratteli-Vershik representation. 
In the context of minimal Cantor systems of finite rank $d$ we will consider representations verifying:
\smallskip

(H1) The entries of $h_{1}$ are all equal to $1$.

(H2) For every $n\geq 2$, $M_{n}>0$.

(H3) For every $n\geq 2$, $d_{n}$ is equal to $d$.

(H4) For every $n\geq 2$, all maximal edges of $E_n$ start in the same vertex of $V_{n-1}$. 
\medskip

A Bratteli-Vershik representation of a  minimal Cantor system $(X,T)$ verifying (H1), (H2), (H3) and (H4) will be called \emph{proper}. In this case, to simplify notations and avoid the excessive use of indexes, we will identify $V_{n}$ with $\{1,\ldots, d\}$ for all $n\geq 1$. The level $n$ will be clear from the context.

It is not difficult to prove that a  minimal Cantor system of finite rank $d$ has a proper representation. We give a brief outline for completeness. We start from a given Bratteli-Vershik representation that we transform by telescoping. Condition (H1) follows by splitting the first level to separate all arrows in the hat and then duplicating accordingly the arrows of the second level. By minimality there is a telescoping of the diagram such that (H2) holds \cite{hps}. Another telescoping to the levels where $\#V_{n}=d$ produces (H3). Property (H4) follows from a compactness argument and a series of telescopings: if this is not possible, then we can construct two disjoint maximal points and we get a contradiction.
\medskip

A minimal Cantor system is \emph{linearly recurrent} if it admits a proper Bratteli-Vershik representation such that the set $\{M_{n};n\geq 1\}$ is
finite. Clearly, linearly recurrent minimal Cantor systems are of finite rank (see \cite{dhs}, \cite{du1}, \cite{du2} and \cite{lr} for more details on this class of systems).

\subsubsection{Associated Kakutani-Rohlin partitions}
Let $B=\left( V,E,\preceq \right)$ be a properly ordered Bratteli diagram and $(X,T)$ the associated minimal Cantor system. This diagram defines for each $n\geq 0$ a clopen {\it Kakutani-Rohlin} partition of $X$: for $n=0$, 
$\P_{0}=\{B_{0}(v_{0})\}$, where $B_{0}(v_{0})=X$, and for $n\geq 1$    
$$
\P_{n}=\{T^{-j}B_{n}(t); t \in V_n, \ 0 \leq j < h_n(t) \} \ ,
$$
where $ B_n(t) = [e_1 , \dots ,e_n]$ and $(e_1 , \dots ,e_n)$ is the unique maximal path from $v_0$ to vertex $t \in V_{n}$. For each $t \in V_n$ the set $\{T^{-j}B_n(t); 0 \leq j < h_n(t) \}$ is called the {\it tower} $t$ of $\P_{n}$. 
It corresponds to the set of all paths from $v_0$ to $t\in V_n$ (there are exactly $h_n(t)$ of such paths). 
Denote by $\T_{n}$ the $\sigma$-algebra generated by the partition $\P_{n}$.
The map $\tau_n: X \to V_n$ is
given by $\tau_n(x)=t$ if $x$ belongs to tower $t$ of $\P_{n}$. The entrance time of 
$x$ to $B_n({\tau_n(x)})$ is given by $r_n(x)=\min\{ j\geq 0; T^jx \in B_n({\tau_n(x)}) \}$. 
\smallskip 

For each $x=(x_1,x_{2},\ldots) \in X$ and $n\geq 0$ define the row vector 
$s_n(x)\in \NN^{d_{n}}$, called the {\it suffix vector of order $n$} of $x$, by
$$
s_n(x,t)=\# \{e \in E_{n+1}; x_{n+1} \preceq e, x_{n+1}\not = e, t \hbox{ is the initial vertex of } e\}
$$
at each coordinate $t\in V_n$. 
A classical computation gives for all $n\geq 1$ (see for example \cite{necesariasuficiente})
\begin{align}
\label{eq:formulareturn} r_n(x)= s_0(x)+\sum_{i=1}^{n-1} \la s_i(x),h_{1}P_{i} \ra=s_0(x)+\sum_{i=1}^{n-1} \la s_i(x),h_{i} \ra \ ,
\end{align}
where $\la \cdot, \cdot \ra$ is the euclidean inner product. Observe that under the hypothesis (H1), i.e., $h_{1}=(1,\ldots,1)$, we have $s_0(x)=0$.

\subsubsection{Invariant measures}

Let $\mu$ be an invariant probability measure of the system $(X,T)$ associated to a properly ordered Bratteli diagram 
$B$, like in the previous subsection. 
It is determined by the values assigned to $B_n(t)$ for all $n\geq 0$ and $t \in V_n$. Define the column vector
$\mu_{n}=(\mu_n(t)\ ; \ t \in V_{n})$ with $\mu_n(t)=\mu(B_n(t))$. 
A simple computation allows to prove the following useful relation:
\begin{equation}\label{eq:measure}
\mu_{m}=P_{m,n}\mu_{n}
\end{equation}
for integers $0\leq m < n$. Also, $\mu(\tau_{n}=t)=h_{n}(t) \mu_{n}(t)$ for all $n\geq 1$ and $t\in V_{n}$.

\subsubsection{Clean Bratteli-Vershik representations}

Let $B$ be a proper ordered Bratteli diagram of finite rank $d$ and $(X,T)$ the corresponding minimal Cantor system. 
Recall that in this case we identify $V_{n}$ with $\{1,\ldots, d\}$ for all $n\geq 1$. 
Then, by Theorem 3.3 in \cite{bkms}, there exist a telescoping of the diagram (which keeps the diagram proper) 
and $\delta >0$ such that:  
\begin{enumerate}
\item For any ergodic measure $\mu$ there exists $I_{\mu} \subseteq \{1,\ldots,d\}$ verifying: 
\begin{enumerate}
\item $\mu(\tau_{n}=t) \geq \delta$ for every $t \in I_{\mu}$ and $n \geq 1$, and   
\item $\lim_{n\to \infty} \mu(\tau_{n}=t)=0$ for every $t \not \in I_{\mu}$.
\end{enumerate}
\item If $\mu$ and $\nu$ are different ergodic measures then $I_{\mu}\cap I_{\nu}=\emptyset$.  
\end{enumerate}

When an ordered Bratteli diagram verifies the previous properties we say it is \emph{clean}. 
We remark that this is a modified version of the notion of \emph{clean} Bratteli diagram given in 
\cite{rangofinito} that is inspired by the results of \cite{bkms}. This property will be very 
relevant for formulating our main result. In \cite{bkms}, systems such that $I_{\mu}=\{1,\ldots,d\}$ 
for some ergodic measure 
$\mu$ are called of \emph{exact finite rank}. Those systems are uniquely ergodic. 
\smallskip

Let $\lambda\in \mathbb{S}^{1}$ be an eigenvalue of the system $(X,T)$ associated to $B$ for an ergodic measure $\mu$. Let $f\in L^{2}(X,\mu)$ be an associated eigenfunction with $|f|=1$. 
For $n\geq 1$ define $c_n:V_n\to \RR_0^{+}$ and $\rho_n:V_n\to [0,1)$ by the relation
\begin{equation}\label{eq:def_c_y_ro}
    \frac{1}{\mu_n(t)}\int_{B_n(t)}f \, d\mu=c_n(t)\lambda^{-\rho_n(t)}, \textrm{\ \; for\ \; } t\in V_n.
\end{equation}
Notice that $0\leq c_{n}(t) \leq 1$. 

The sequence $(f_n;n\geq 1)$ of conditional expectations of $f$ with respect to the sigma algebras $(\T_{n}; n\geq 1)$ generated by the Kakutani-Rohlin partitions satisfies
$$f_n(x)=\econd{f}{\mathcal{T}_{n}}(x)=c_n(\tau_n(x))\lambda^{-r_n(x)-\rho_n(\tau_n(x))}.$$
It can be proved that   $\lambda^{-(r_n+\rho_n\circ\torref{n})}$ converges $\mu$--a.e. (for a slightly deeper discussion we refer the reader to \cite{necesariasuficiente}).
Also, rephrasing a known result from \cite{rangofinito} we have

\begin{lemm}\label{lemm:c_tiende_a_uno} 
If $B$ is a clean Bratteli diagram and $\mu$ an ergodic measure for the associated minimal Cantor system, 
then
\begin{enumerate}
\item for any $t \in \{1,\ldots , d\}$, $\lim_{n\to \infty}\mu(\torref{n}=t)(c_n(t)-1)\to 0$,  
\item for $t\in I_{\mu}$, $\lim_{n\to \infty} c_n(t)\to 1$.  
\end{enumerate}
\end{lemm}

\subsection{Bratteli-Vershik systems of Toeplitz type}

A properly ordered Bratteli diagram $B=(V,E,\preceq)$ is of \emph{Toeplitz type} if for all $n\geq 1$ 
the number of edges in $E_{n}$ finishing at a fixed vertex of $V_{n}$ is constant independently of the vertex. 
Denote this number by $q_n$ and set $p_n = q_1 q_{2} \cdots q_n$. Observe that $p_{n}$ is the number of paths from $v_{0}$ to any vertex of  $V_{n}$.
Thus $h_{n}(t)=p_{n}$ for any $t \in V_{n}$. We say that $(q_n;n\geq 1)$ is the \emph{characteristic sequence} of the diagram.
This class was obtained in \cite{gjtoeplitz} when characterizing Toeplitz
subshifts.  

The main object in this study are eigenvalues of minimal Cantor systems of finite rank $d$, having a proper Bratteli-Vershik representation of Toeplitz type. It is known that finite rank minimal Cantor systems are either odometers or subshifts \cite{DM}, so in our study we will be dealing only with Toeplitz subshifts or odometers.

To state our main results we will need some extra notations. Fix a minimal Cantor system $(X,T)$ with a Toeplitz type proper Bratteli-Vershik representation of rank $d$ and characteristic sequence $(q_n;n\geq 1)$.

For $0 \leq m < n$ define $q_{m,n}=q_{m+1}\cdots q_{n}$, 
the number of paths in $E_{m,n}$ finishing 
in any fixed vertex $t\in V_n$. Clearly $q_{\ell,n}=q_{\ell,m}q_{m,n}$ if $0\leq \ell < m <n$.
Also, for $x=(x_{1}, x_{2}, \ldots) \in X$ define the integer $\sufijofbb{m}{n}(x)$ as the number of paths
in $E_{m,n}$ which end at $\torref{n}(x)$ that are strictly bigger than
$(x_{m+1},\ldots, x_{n})$ with respect to the induced partial order in $E_{m,n}$. 
Finally, define the set $\sufijocbb{m}{n}{t_{1}}{t_{2}}$ for $t_{1}\in V_m$ and $t_{2}\in
V_n$ by $$\sufijocbb{m}{n}{t_{1}}{t_{2}}=\set{\sufijofbb{m}{n}(x)\talque \tau_m(x)=t_{1} \y \tau_n(x)=t_{2}}.$$
It is not difficult to prove that the cardinality of $\sufijocbb{m}{n}{t_{1}}{t_{2}}$ is equal to 
$\matrizp{m}{n}{t_{1}}{t_2}$, that is, the number of paths from $t_{1} \in V_{m}$ to $t_{2} \in V_{n}$.

If necessary, to simplify notations we will denote $\sufijocbb{n}{n+1}{{t_{1}}}{t_{2}}$ by 
$\sufijocb{n}{t_{1}}{t_{2}}$ and $\sufijofbb{n}{n+1}$ by $\sufijofb{n}$. 
Notice that $\sufijofb{n}(x)=\la s_{n}(x), (1,\ldots,1) \ra=\sum_{t\in V_{n}} s_{n}(x,t)$ for any $x\in X$.
\medskip

We will need the following simple relations.
For $0\leq \ell < m < n$, $ t_{1}\in V_{\ell}$ and $x\in X$ the following equalities hold:
\begin{eqnarray}
  \label{eq:retorno_toeplitz} r_\ell(x) &=& \sufijofb{0}(x) + \sum_{i=1}^{\ell-1}p_i\sufijofb{i}(x), \\
  \label{eq:sufijo_dos_niveles}\sufijofbb{\ell}{m}(x) &=& \sufijofb{\ell}(x)+\sum_{i=\ell+1}^{m-1}q_{\ell+1}q_{\ell+2}\cdots q_{i}\sufijofb{i}(x) \nonumber \\
  &=& \frac{r_m(x)-r_\ell(x)}{p_\ell},  \\
  \label{eq:interpolacion_sufijos} \sufijofbb{\ell}{n}(x) &=& \sufijofbb{\ell}{m}(x)+q_{\ell,m}\sufijofbb{m}{n}(x),
\end{eqnarray}
\begin{eqnarray}
  \label{eq:base_en_bases_superiores} B_{\ell}(t_{1})&=&\bigcup_{t_{2}\in V_{m}}\bigcup_{s\in\sufijocbb{\ell}{m}{t_{1}}{t_{2}}}T^{-p_\ell s}B_{m}(t_{2}),
\end{eqnarray}
where the union in the right hand side is disjoint.

\section{Eigenvalues of Toeplitz systems of finite rank}

As was mentioned in the introduction, any countable subgroup of $\mathbb{S}^1=\{z \in \CC \ ; |z|=1 \}$ containing infinitely many rationals can be the set of eigenvalues of a Toeplitz subshift for a given invariant measure \cite{iwanik,toeplitzpreset}.
Also, $\exp(2 i \pi \, \alpha)  \in \mathbb{S}^{1}$ is a continuous eigenvalue of a minimal Cantor system with a Toeplitz type proper Bratteli-Vershik representation
if and only if $\alpha = a/p_m$ for some $a \in \ZZ$ and $m\geq 1$ \cite{Williams,Keane}. 
A direct proof can be given using the particular combinatorial structure of the Brattelli-Vershik representation 
of a  minimal Cantor system of Toeplitz type.  We sketch it here.
Using  \eqref{eq:retorno_toeplitz} and the fact that $p_{m}$ divides $p_{n}$ when $m\leq n$, one gets that
$r_{n}(x) /p_{m} = ({\bar s}_{0}(x) +\sum_{i=1}^{m-1} p_{i}{\bar s_{i}}(x))/p_{m} \mod \ZZ$, for all $n\geq m$. 
Hence, $\exp(2i\pi r_{n}(x) /p_{m})$ converges uniformly when $n\to \infty$, which is a necessary and sufficient condition for 
$\exp(2i\pi /p_{m})$, and thus $\exp(2 i \pi \ a/p_{m})$ for every $a \in \ZZ$, to be continuous eigenvalues in this context (see Proposition 12 in \cite{necesariasuficiente}). 

In the opposite direction, using the same criterion, if 
$\exp(2i\pi /b)$ with $b \in \ZZ$ is a continuous eigenvalue, then   
$(r_{n+1}(x)-r_{n}(x))/b=p_{n}{\bar s}_{n}(x)/b \mod \ZZ$ is close to $0$  
for any large enough $n \geq 1$ and uniformly in $x$. Taking a point $x$ such that ${\bar s}_{n}(x)=1$ allows to conclude that $1/b=a/p_{n}$ for some large $n\geq 1$ and $a\in \ZZ$. More details about continuous eigenvalues of Toeplitz type Bratteli-Vershik systems can be found in \cite{rangofinito}.

In the class of minimal Cantor systems with a Toeplitz type representation, the assumption of finite topological rank restricts the possibilities for non continuous eigenvalues. But, importantly, all are rational. In addition, if the characteristic sequence of a proper representation is bounded (or equivalently, a proper representation gives a linearly recurrent system), then all the eigenvalues are continuous. The following theorem gives a very restrictive condition verified by non continuous eigenvalues of Toeplitz systems in the finite rank case that are not linearly recurrent.

\begin{theo}{\cite{rangofinito}} 
\label{theo:vpsonracionales}
Let $(X,T)$ be a minimal Cantor system with a Toeplitz type proper Bratteli-Vershik representation of rank $d$ and characteristic sequence $(q_n;n\geq 1)$. Let  $\mu$ be an ergodic probability measure. If $\exp(2i \pi \, a/b)$, with $(a,b)=1$, is a non continuous
rational eigenvalue of $(X,T)$ for $\mu$, then $b/(b,p_n) \leq d$ for all $n$ large enough.
\end{theo}

Let $\ld=\exp(2 i\pi \, a/b)$, with $a,b$ integers such that $(a,b)=1$, be a non continuous rational eigenvalue as in the previous theorem. We notice that $b/(b,p_n) >1$ for all $n$ large enough. Indeed, if $b/(b,p_n)=1$ for some $n\geq 1$, then 
$1/b=a'/p_{n}$ for some $a' \in \ZZ$, which by the discussion above implies that $\exp(2i\pi \, a/b)$ is a continuous eigenvalue. Also, observe that $(b,p_n)$ is a non decreasing sequence of integers bounded by $b$, so $b/(b,p_n)$ is eventually constant, say equal to $\rb$. Since we are considering proper representations, the fact that $\rb>1$ implies that $(q_n;n\geq 1)$ tends to infinity with $n$. Otherwise, the system will be linearly recurrent, and thus all eigenvalues will be continuous, which implies that  $b/(b,p_n)=1$ for some $n>1$. 
\medskip

Now we state our main result, 
\medskip

\begin{theo}
\label{theo:principal}
Let $(X,T)$ be a minimal Cantor system with a Toeplitz type proper and clean Bratteli-Vershik representation of rank $d$ and  characteristic sequence $(q_n;n\geq 1)$. Let $\mu$ be an ergodic probability measure.
Then, $\lambda=\exp(2i\pi a/b)$, with $a,b$ integers such that $(a,b)=1$, is a non continuous eigenvalue of $(X,T)$ for $\mu$  if and only if 
\begin{enumerate}
\item
$b/(b,p_n)=\rb$ for all $n$ large enough and some $1<\rb\leq d$, and
\item\label{mainsecond}
for all $t_2\in I_{\mu}$ 
$$\sum_{t_{1}\in V_{m}}\frac{\abs{\sum_{s\in\sufijocbb{m}{n}{t_{1}}{t_2}}\lambda^{-p_ms}}}{q_{m,n}}\xrightarrow[m,n\to\infty]{}1,$$ 
uniformly in $m,n \in \NN$ with $m<n$.
\end{enumerate}
\end{theo}
As was mentioned in the introduction, even if this condition looks ``heavy'' to check, in fact it is easy to verify and construct examples fulfilling it. This will be illustrated in Section \ref{examples}. The main tool is provided by the following corollary that follows from the construction in the proof of Theorem \ref{theo:principal}.
\smallskip

\begin{coro}
\label{coro:cocycle}
Let $(X,T)$ be a minimal Cantor system with a Toeplitz type proper and clean Bratteli-Vershik representation of rank $d$ 
and  characteristic sequence $(q_n;n\geq 1)$. Let $\mu$ be an ergodic probability measure.
Let $(q_n;n\geq 1)$ be its characteristic sequence. 
Then, $\lambda=\exp(2i\pi a/b)$, with $a,b$ integers such that $(a,b)=1$, is a non continuous eigenvalue of $(X,T)$ for $\mu$  if and only if up to a telescoping of the diagram we have
\begin{enumerate}
\item
$p_{n} = p \mod b$ for some $p \in \{0,\ldots,b-1\}$ and for all $n\geq 2$,
\item
$b/(b,p_n)=\rb$ for all $n$ large enough and some $1<\rb\leq d$, 
\item
there exists a map 
$k(\cdot,\cdot): \{1,\ldots,d \}\times \{1,\ldots,d \} \to \{0,\ldots,\rb-1\}$ 
such that
\begin{align*}
p \cdot k (t_{1},t_{3}) &=  p \cdot k (t_{1},t_{2}) + p \cdot k (t_{2},t_{3}) \mod b, \\
p \cdot k (t_{1},t_{1}) &= 0 \mod b, \ \  p \cdot k (t_{1},t_{2}) = -p \cdot k (t_{2},t_{1}) \mod b,
\end{align*}
for all $t_{1},t_{2}, t_{3}\in I_{\mu}$, 
\item for $\mu$--almost every point $x\in X$ the equality $\bar s_{n}(x)=k(\tau_{n}(x),\tau_{n+1}(x)) \mod \rb$ holds 
for all large enough $n\in \NN$.
\end{enumerate}
\end{coro}

In what follows we provide a number of reformulations and corollaries of the main theorem. 
Some proofs are left to the reader since they can be easily deduced from a direct computation or Lemmas \ref{lemm:suma_matriz_coef_equivalentes} and \ref{lemm:indice_nivelmenor_sobre} provided below, others will be proved near the end of Section \ref{mainproofs} after proving the main theorem.

We start by a natural reformulation of Theorem \ref{theo:principal}. It says that we can replace $V_{m}$ by $I_{\mu}$ in the sum of statement (2) of the theorem. In other words, we only need to consider the vertices of the diagram determining the measure $\mu$. 
We will need the following observation: for $t_{1}\not\in I_{\mu}$ and $t_2\in I_{\mu}$ one has
\begin{equation}\label{eq:p_tiende_a_cero_fuera_de_I}
\frac{\matrizp{m}{n}{t_{1}}{t_2}}{q_{m,n}}\xrightarrow[m,n\to\infty]{}0
\end{equation}
uniformly in $m,n \in \NN$ with $m<n$. Indeed, since the diagram is clean, 
$\mu(\torref{n}=t_2) \geq \delta > 0$ and $\lim_{m\to \infty}\mu(\torref{m}=t_1)=0$. These facts, together with the following inequalities  
$$\frac{\matrizp{m}{n}{t_{1}}{t_2}}{q_{m,n}} \cdot \delta \leq
\frac{\matrizp{m}{n}{t_{1}}{t_2}}{q_{m,n}}\mu(\torref{n}=t_2)=\mu(\torref{m}=t_1, \torref{n}=t_{2})\leq \mu(\torref{m}=t_{1}),
$$ 
allow to deduce \eqref{eq:p_tiende_a_cero_fuera_de_I}. Since the cardinality of $\sufijocbb{m}{n}{t_{1}}{t_2}$ is equal to 
$\matrizp{m}{n}{t_{1}}{t_2}$, we also deduce that
$$\sum_{t_{1}\in V_{m}\setminus I_{\mu}}\frac{\abs{\sum_{s\in\sufijocbb{m}{n}{t_{1}}{t_2}}\lambda^{-p_ms}}}{q_{m,n}}
\leq \sum_{t_{1}\in V_{m}\setminus I_{\mu}}\frac{\matrizp{m}{n}{t_{1}}{t_2}}{q_{m,n}}.$$

Therefore, a direct application of \eqref{eq:p_tiende_a_cero_fuera_de_I} in the last inequality allows to reformulate Theorem \ref{theo:principal} as follows.

 \begin{coro}[Variation on Theorem \ref{theo:principal}]\label{theo:principal_var1} 
The complex number $\lambda=\exp(2i\pi a/b)$, with $a,b$ integers such that $(a,b)=1$, is a non continuous eigenvalue of $(X,T)$ for $\mu$  if and only if 
\begin{enumerate}
\item
$b/(b,p_n)=\rb$ for all $n$ large enough and some $1<\rb\leq d$, and
\item
for all $t_2\in I_{\mu}$  
$$\sum_{t_{1}\in I_{\mu}}\frac{\abs{\sum_{s\in\sufijocbb{m}{n}{t_{1}}{t_2}}\lambda^{-p_ms}}}{q_{m,n}}\xrightarrow[m,n\to\infty]{}1,$$ 
uniformly in $m,n \in \NN$ with $m<n$.
\end{enumerate}
\end{coro}
\bigskip

The following corollary is a reformulation of the main condition of Theorem \ref{theo:principal} and the corresponding one in Corollary \ref{theo:principal_var1}. It follows almost directly by combining Lemmas \ref{lemm:suma_matriz_coef_equivalentes} and \ref{lemm:indice_nivelmenor_sobre} in the next section, so its proof is left to the reader.

\begin{coro}\label{cor:condwithpartition}
The main condition in Theorem \ref{theo:principal} (resp. Corollary \ref{theo:principal_var1}) is equivalent to: 
for all $t_2\in I_{\mu}$ and $m\geq 1$ there exists a sequence of partitions $(\H_{m,n,t_2}; m<n)$ of 
$V_{m}$ (resp. of $I_{\mu}$) with $\#\H_{m,n,t_2}=\rb$ such that
$$\sum_{t_{1}\in A}\frac{\abs{\sum_{s\in\sufijocbb{m}{n}{t_{1}}{t_2}}\lambda^{-p_ms}}}{q_{m,n}}\xrightarrow[m,n\to\infty]{}\frac{1}{\rb},$$ 
uniformly in $m,n \in \NN$ with $m<n$ for any $A\in\H_{m,n,t_2}$.
\end{coro}
\medskip

This formulation pinpoints to the possible local orders that accept a Bratteli-Vershik representation to have non continuous eigenvalues. Part (3) of Lemma \ref{lemm:suma_matriz_coef_equivalentes} states that the main condition of Theorem \ref{theo:principal} (or its equivalent formulations) implies that the local order of most of the arrows from a vertex in an atom $A \in \H_{m,n,t_2}$ to $t_{2} \in I_{\mu}$ at level $n$ must be congruent modulo $\rb$. 
This condition is one of the main tools to explore non continuous rational eigenvalues of Toeplitz systems.   
\medskip

Another interesting fact is that we can relate non continuous eigenvalues with the number of ergodic invariant measures of a Toeplitz system. Let $(X,T)$ be a minimal Cantor system and $\mu$ an ergodic measure as in Theorem \ref{theo:principal}.
Define,
$$
\rB_{\mu}=\{\lim_{m\to \infty}b/(b,p_m) ;  b \in \NN, \ \exp(2i\pi/b)   \textrm{ is a non continuous eigenvalue for } \mu \}
$$
and endow it with the divisibility (partial) order. Recall that $\lim_{m\to \infty}b/(b,p_m)$ is equal to $\rb=b/(b,p_n)$ for a large $n \in \NN$. 
Denote by $\mathcal{M}_{erg} (X,T)$ the set of ergodic measures of $(X,T)$ and consider the set $\mathcal{M}$ defined by:
$$
\mathcal{M}=\left\{\mu \in \mathcal{M}_{erg}(X,T) \ ; \  \rB_{\mu}\not = \emptyset \right \}.
$$

\begin{coro}
\label{cor:alternative} The following properties hold:
\begin{enumerate}
\item For any $\mu \in \mathcal{M}$ and $\rb \in \rB_{\mu}$, $\rb \leq \# I_{\mu}$.
\item For any $\mu\in \mathcal{M}$, $\rB_{\mu}$ has a unique divisibility-maximal element $\rb_{\mu}$.
\item
$\sum_{\mu \in \mathcal{M}} \rb_{\mu}\leq d$.
\item
$\# \mathcal{M} \leq \# \mathcal{M}_{erg} (X,T)\leq d-\sum_{\mu \in \mathcal{M}} (\rb_{\mu}-1)$.
\end{enumerate}
\end{coro}
\medskip
The proof of this corollary will be given at the end of Section \ref{mainproofs}.
\medskip

Fix an ergodic measure $\mu$. To understand better the last corollary let us suppose the $p_{n}$'s are powers of the same prime number. In this case, for all integers $b$ such that $\lambda=\exp(2i\pi/b)$ is a non continuous eigenvalue for $\mu$ one has $(b,p_{n})=1$  and parts (1) and (2) of last corollary tell us that there is a unique 
$b=\rb_{\mu} \leq \#I_{\mu} \leq d$ which is maximal in $\rB_{\mu}$. All other non continuous eigenvalues for $\mu$ 
are powers of $\lambda$.  If $\rB_{\mu}$ is empty, no non continuous eigenvalues exist for $\mu$. Notice that property (1) implies that we need at least $\rb_{\mu}$ vertices to have the non continuous eigenvalue $\lambda$. Since $I_{\mu} \cap I_{\nu}=\emptyset$ for different ergodic measures $\rb_{\nu} \leq d-\#I_{\mu} \leq d -\rb_{\mu}$. We will see in some examples of Section \ref{examples} that these inequalities can be strict. 
\medskip

In the particular case when $\rb_{\mu}=d$ for some ergodic measure $\mu$ we get the following corollary.

\begin{coro}
\label{coro:ue}
Consider $\ld=\exp{(2 i \pi a/b)}$, with $a,b$ integers such that $(a,b)=1$ and $b/(b,p_n)=d$ for all $n$ large enough. Then $\ld$ is a non continuous eigenvalue of $(X,T)$ for the invariant measure $\mu$ if and only if
for all $t_1, t_{2} \in \{1,\ldots,d\}$

\begin{equation}\label{eq:coro7}
\frac{\abs{\sum_{s\in\sufijocbb{m}{n}{t_1}{t_2}}\lambda^{-p_ms}}}{q_{m,n}}\xrightarrow[m,n\to\infty]{}\frac{1}{d}
\end{equation}
uniformly in $m,n \in \NN$ with $m<n$.
If $\lambda$ is an eigenvalue, then:
\begin{enumerate}
\item\label{uergodic}
the system $(X,T)$ is uniquely ergodic and $\mu$ is the unique invariant measure,
\item
\label{theo:principal1}
for all $t\in \{1,\ldots, d\}$, $\ds\lim_{n\to\infty}\mu(\tau_n=t)=1/d$.
\end{enumerate}
\end{coro}

Condition \eqref{eq:coro7} and statement \eqref{uergodic} follow almost directly from Corollaries \ref{cor:condwithpartition} and \ref{cor:alternative}. Nevertheless, we provide a complete proof of the corollary at the end of Section \ref{mainproofs}.
\medskip

An analogous result to Corollary \ref{coro:ue} can be obtained
when the system is uniquely ergodic and $b/(b,p_n)=\#I_{\mu}$ for all $n$ large enough. The statement is obtained by replacing $d$ by $\#I_{\mu}$ and the set $\{1,\ldots,d\}$ by $I_{\mu}$ in the last corollary.

\section{Main technical lemmas}
In this section we will provide the main ingredients we need to prove Theorem \ref{theo:principal} and its corollaries.

\subsection{A geometric lemma}

The next lemma can be stated in a much more general situation and its proof follows from general facts of convex analysis. Nevertheless, since we consider a particular case, we provide a simple self-contained proof. 

\begin{lemm}\label{lemm:geometrico}
Let $N$ be a positive integer. Then there exists a constant $C$ such that for any convex combination 
$w=\sum_{j=0}^{N-1} \alpha_{j}\xi^{j}$ of the $N$-th roots of unity $1,\xi,\ldots,\xi^{N-1}$ verifying $1-\varepsilon<\abs{w}\leq 1$ for some $\varepsilon >0$ one has
$$1-C\varepsilon< \alpha_{i}\leq 1$$ for some $0\leq i\leq N-1$.
\end{lemm}
\begin{proof}
A proof is given only in the case when $\abs{w}\neq 1$. Write $w$ in the following way $$w=\alpha_{i}\xi^{i}+\beta\zeta,$$ where
$\alpha_{i}\geq 1/N$ and $\alpha_{i}+\beta=1$ (note that $\zeta$ belongs to the convex hull of the $N$-th roots of unity different from
$\xi^{i}$).
The function $F(z)=\alpha_{i}\xi^{i}+\beta z$ has maximal absolute value at $z\in \{\xi^{i-1}, \xi^{i+1}\}$ when restricted to the convex
hull of the $N$-th roots of unity different from $\xi^{i}$. Hence

\begin{eqnarray*}
  1-\varepsilon &<& \abs{w}\ \; \; (=\abs{F(\zeta)}) \\
   &\leq& \abs{F(\xi^{i+1})} \\
   &= & \abs{1+\beta(\xi-1)} \\
   &= & \sqrt{1-2\beta(1-\beta)(1-\cos{2\pi/N})}\\
   &\leq & {1-\beta(1-\beta)(1-\cos{2\pi/N})}\\
   &\leq& 1-\beta\left(\frac{1-\cos(2\pi/N)}{N}\right)
\end{eqnarray*}
and $$\alpha_{\ell}>1-\left(\frac{N}{1-\cos(2\pi/N)}\right)\varepsilon.$$
\end{proof}

\subsection{Special telescoping of a Bratteli-Vershik system}

At some point of the proof of Theorem \ref{theo:principal} we will need to telescope an ordered Bratteli diagram in the following particular way. 
 
\begin{lemm}
\label{lemma:contraction-partition}
Let $B=\left(V,E,\preceq \right)$ be an ordered Bratteli diagram such that 
$\#V_n = d$ for all $n\geq 1$ and identify $V_{n}$ with $\{1,\ldots,d \}$. 
For all $1\leq m<n $ and $t\in \{1,\ldots,d \}$ consider $(\mathcal{G}_{m,n,t},\leq_{m,n,t})$, where 
$\mathcal{G}_{m,n,t}$ is a partition of $V_m$ and $\leq_{m,n,t} $ is a total ordering on the atoms of  
$\mathcal{G}_{m,n,t}$.
Then, there exists a strictly increasing sequence $(n_k)_{k\geq 0}$ in $\NN$
such that for all $k_0\geq 0$, $k> k_0$ and $t\in \{1,\ldots,d \}$ we have 
$$(\mathcal{G}_{n_{k_0},n_k,t} ,\leq_{n_{k_0},n_k,t})  = (\mathcal{G}_{n_{k_0},n_{k+1},t} , \leq_{n_{k_0},n_{k+1},t} ).$$
\end{lemm}

\begin{proof}
It suffices to remark that there are finitely many such structures on $\{ 1, \ldots, d \}$
(partitions endowed with total orderings). Then, one proceeds by induction using the pigeon hole principle. 

Let us give some details. Take $n_0 = 1$.
By the pigeon hole principle, there exists a strictly increasing sequence $(n^{(0)}_k)_{k\geq 0}$, 
with $n^{(0)}_0>n_0$, such that for all $k\geq 0$ and  $t\in \{ 1, \ldots , d\}$ we have 
$$(\mathcal{G}_{n_0,n^{(0)}_k ,t} , \leq_{n_0,n^{(0)}_k ,t}) = (\mathcal{G}_{{n_0},n^{(0)}_{k+1},t} , \leq_{{n_0},n^{(0)}_{k+1},t} ).$$

Now, let $n_1 = n^{(0)}_0$.
Using the same argument, there exists a strictly increasing subsequence $(n^{(1)}_k)_{k\geq 0}$ of $(n^{(0)}_k)_{k\geq 0}$, with $n^{(1)}_0>n^{(0)}_0$, such that for all $k\geq 0$ and $t\in \{ 1, \ldots , d\}$ we have $(\mathcal{G}_{n_1,n^{(1)}_k , t} , \leq_{n_1,n^{(1)}_k , t} )  = (\mathcal{G}_{{n_1},n^{(1)}_{k+1},t} , \leq_{{n_1},n^{(1)}_{k+1},t} )$.
Observe that we also have $(\mathcal{G}_{n_0,n^{(1)}_k , t} , \leq_{n_0,n^{(1)}_k , t} ) = 
(\mathcal{G}_{{n_0},n^{(1)}_{k+1},t} , \leq_{{n_0},n^{(1)}_{k+1},t} ) $ for all $k \geq 0$ and $t\in \{ 1, \ldots , d\}$ by construction.
Proceeding in this way we obtain the desired sequence $(n_k)_{k\geq 0}$.
\end{proof}

\subsection{Uniform lower bound for consecutive  towers in $I_{\mu}$}

\begin{lemm}
\label{lemma:low-independence}
Let $(X,T)$ be a minimal Cantor system with a Toeplitz type proper and clean Bratteli-Vershik representation of rank $d$ and  $\mu$ be an ergodic probability measure. Let $(q_n;n\geq 1)$ be its characteristic sequence. For all $m\geq 1$, 
there exists $n_0 > m$ such that for all $n\geq n_0$ and $t_1,t_2\in I_{\mu}$ 
$$
\frac{\matrizp{m}{n}{t_{1}}{t_{2}}}{q_{m,n}} \geq \frac{\delta}{3},
$$
where $\delta >0$ is such that $\mu(\tau_{n}=t)\geq \delta$ for any $t \in I_{\mu}$ and $n \in \NN$ (coming from the cleanliness property of the diagram). 
\end{lemm}
\begin{proof}
Fix $m\geq 1$ and $0< \epsilon < \delta^2/3$.
From Egorov's theorem and the ergodic theorem, there exists a measurable subset 
$A_\epsilon$ with $\mu (A_\epsilon )\geq 1-\epsilon$ and a positive integer $M_0$ such that 
for all $x\in A_\epsilon$ and $M\geq M_0$ we have 

\begin{equation}
\label{Egorov}
\left |\frac{1}{M}\sum_{k=0}^{M-1} 1_{\{ \tau_m = t_1\}} (T^k x) - \mu ( \tau_m = t_1) \right |  <\epsilon .
\end{equation}

Let $n > m$ be such that $p_n\geq M_0$ (recall that $p_{n}$ is the number of paths from $v_{0}$ to any vertex of $V_{n}$). There exists $x\in A_\epsilon \cap T^{-p_n-j+1} B_n(t_2)$ for some $0\leq j\leq \lfloor\frac{\epsilon p_n}{\delta}\rfloor < p_{n}$.
Indeed, 
$$\mu\left(\bigcup_{j=0}^{\lfloor\frac{\epsilon p_n}{\delta} \rfloor} T^{-(p_{n}+j-1)}B_n({t_{2}})\right )=
\left (\left \lfloor\frac{\epsilon p_n}{\delta} \right \rfloor+1 \right ) \mu(B_n({t_2}) ) > \frac{\epsilon}{\delta} \mu(\tau_{n}=t_{2}) \geq \epsilon,$$
since $\mu(\tau_{n}=t_{2})= p_{n}\mu(B_n({t_2}))$ and $t_{2} \in I_{\mu}$. Hence, 
$\bigcup_{j=0}^{\lfloor\frac{\epsilon p_n}{\delta} \rfloor} T^{-(p_{n}+j-1)}B_n({t_{2}})$ must intersect $A_{\epsilon}$. 
Notice that the iterates $T^{j}x,\ldots, T^{j+p_{n}-1}x$ cross completely tower $t_{2} \in V_{n}$, from the lowest to the highest level. So those iterates enter to tower $t_{1} \in V_{m}$ exactly ${\matrizp{m}{n}{t_{1}}{t_{2}}} p_m$ times.

Then, since $t_{1} \in I_{\mu}$, $p_{n}+j \geq M_{0}$ and $x \in A_{\epsilon}$, we can use \eqref{Egorov} to get 
\begin{align*}
\delta - \epsilon & \leq \mu ( \tau_m = t_1 ) - \epsilon \leq
\frac{1}{p_n+j}\sum_{k=0}^{p_n+j-1} 1_{\{ \tau_m = t_1\}} (T^k x) \\
& \leq \frac{j}{p_n+j} + \frac{1}{p_n+j}\sum_{k=0}^{p_n-1} 1_{\{ \tau_m = t_1\}} (T^k (T^{j}x))\\
& \leq \frac{\epsilon}{\delta} + \frac{{\matrizp{m}{n}{t_{1}}{t_{2}}} p_m}{p_n+j} 
\leq 
\frac{\epsilon}{\delta} + \frac{{\matrizp{m}{n}{t_{1}}{t_{2}}} }{q_{m,n}} \leq
\frac{\delta}{3}+ \frac{{\matrizp{m}{n}{t_{1}}{t_{2}}} }{q_{m,n}},
\end{align*}
which ends the proof.
\end{proof}

\subsection{Equivalent conditions for Theorem \ref{theo:principal}}

We follow the same notations as in Theorem \ref{theo:principal}: $\lambda=\exp(2i\pi a/b)$, with $(a,b)=1$,  
and $\rb$ is the limit in $n$ of $b/(b,p_n)$, which is attained from some large $n \in \NN$. 
In the sequel, equality modulo $\rb$ and $b$ will be written $\equiv_\rb$ and $\equiv_b$ respectively.

To make  the text lighter, we need to introduce some extra notations. 
For $t_1, t_{2}\in \{1,\ldots,d\}$, $k \in \{0,\ldots,\rb-1\}$ and integers $1\leq m <n$, set
\medskip

\begin{eqnarray}
\label{eq:def_suma}\suma{m}{n}{t_1}{t_2} &=& \sum\nolimits_{s\in\sufijocbb{m}{n}{t_1}{t_2}}\ld^{-p_ms} \textrm{\ \; \; \; }\\
   \label{eq:def_coef}\sumacoef{m}{n}{t_1}{t_2}{k} &=& \#\set{s\in\sufijocbb{m}{n}{t_1}{t_2}\talque \ s\equiv_{\rb} k }
\end{eqnarray}
\medskip

Notice that for $s,s'\in\sufijocbb{m}{n}{t_1}{t_2}$, $\ld^{-p_m s}=\ld^{-p_m s'}$ if and only if $s\equiv_{\rb} s'$. Then,
\begin{eqnarray}
  \label{eq:suma_en_sumacoef}\suma{m}{n}{t_1}{t_2} &=& \sum_{k=0}^{\rb-1}\ld^{-p_m k}\sumacoef{m}{n}{t_1}{t_2}{k}, \\
  \label{eq:matriz_en_sumacoef}\matrizp{m}{n}{t_1}{t_2} &=& \sum_{k=0}^{\rb-1}\sumacoef{m}{n}{t_1}{t_2}{k}, \\
  \label{eq:desigualdad_suma_matriz}\abs{\suma{m}{n}{t_1}{t_2}}&\leq & \matrizp{m}{n}{t_1}{t_2}, \\
  \label{eq:descomposicion_qmn}q_{m,n} &=& \sum_{k=0}^{\rb-1}\sum_{t_{1}\in V_m}\sumacoef{m}{n}{t_{1}}{t_{2}}{k}, \\
  \label{eq:sumacoef_torres_todas}\sum_{t_{1}\in V_m}\sumacoef{m}{n}{t_{1}}{t_2}{k} &=& \left\lfloor\frac{q_{m,n}}{\rb}\right\rfloor \textrm{\ \; or\ \; } \left\lfloor\frac{q_{m,n}}{\rb}\right\rfloor + 1.
\end{eqnarray}
\medskip

\begin{lemm}\label{lemm:suma_matriz_coef_equivalentes}
For any $t_2 \in \{1, \ldots, d\}$ the following conditions are equivalent:

(1) $\ds \sum_{t_{1}\in V_m}\frac{\abs{\suma{m}{n}{t_{1}}{t_2}}}{q_{m,n}}\xrightarrow[m,n\to\infty]{} 1 \textrm{\; \; \;  uniformly in $m,n \in \NN$ with $m<n$\;}$ (this is condition \eqref{mainsecond} of Theorem \ref{theo:principal} stated for any $t_{2}$).

\smallskip

(2) For all $t_{1}\in \{ 1,\ldots, d\}$,
\smallskip

$$\frac{\abs{\suma{m}{n}{t_{1}}{t_2}}}{q_{m,n}}-\frac{\matrizp{m}{n}{t_{1}}{t_2}}{q_{m,n}}\xrightarrow[m,n\to\infty]{} 0 \textrm{\; \; \; \;}$$
uniformly in $m,n \in \NN$ with $m<n$.
\smallskip

(3) For all integers $1\leq m<n$ and $t_{1} \in \{1,\ldots,d\}$, there exists $\indice{k}{m}{n}{t_{1}}{t_2}$ in $\{0,\ldots,\rb-1\}$  such that
\smallskip

$$\frac{\sumacoef{m}{n}{t_{1}}{t_2}{\indice{k}{m}{n}{t_{1}}{t_2}}}{q_{m,n}}-\frac{\matrizp{m}{n}{t_{1}}{t_2}}{q_{m,n}}\xrightarrow[m,n\to\infty]{} 0
\textrm{\; \; \; \;}$$
uniformly in $m,n \in \NN$ with $m<n$.
\end{lemm}
\begin{proof}
(1)$\implica$(2). We proceed by contradiction. Suppose there exists $\overline{t_{1}}\in \{1,\ldots,d\}$ such that
for infinitely many positive integers $m,n$ with $m<n$ 
\begin{equation}\label{badthing}
\frac{\matrizp{m}{n}{\overline{t_{1}}}{t_2}}{q_{m,n}}-\frac{\abs{\suma{m}{n}{{\overline  t_{1}}}{t_2}}}{q_{m,n}}\ge 2\varepsilon >0,
\end{equation}
where $\varepsilon$ is a positive real.

From (1) we have that for any large enough positive integers $m,n$ with $m<n$ 
\begin{equation}
\label{eq:auxiliar}
1-\varepsilon<\sum_{t_{1}\in V_m}\frac{\abs{\suma{m}{n}{t_{1}}{t_2}}}{q_{m,n}}<1+\varepsilon.
\end{equation}
Consider a pair of large integers $m,n$ with $m<n$ verifying \eqref{badthing}.
Then, from \eqref{eq:desigualdad_suma_matriz}, \eqref{badthing} and \eqref{eq:auxiliar} we get
$$
  1 = \sum_{t_{1}\in V_m}\frac{\matrizp{m}{n}{t_{1}}{t_2}}{q_{m,n}}
  \geq  2\varepsilon + \sum_{t_{1}\in V_m}\frac{\abs{\suma{m}{n}{t_{1}}{t_2}}}{q_{m,n}}
  \geq  1+\varepsilon,
$$
which is impossible. Condition (2) follows.
\smallskip

(2)$\implica$(3). 
Take $\varepsilon>0$. By hypothesis and \eqref{eq:desigualdad_suma_matriz}, there exists a positive integer $N$ such that for all $n>m>N$ and $t_{1}\in \{1,\ldots,d\}$
$$0\leq\frac{\matrizp{m}{n}{t_{1}}{t_2}}{q_{m,n}}-\frac{\abs{\suma{m}{n}{t_{1}}{t_2}}}{q_{m,n}}<\varepsilon.$$
Alternatively, the last inequality can be written as
$$1-\frac{\varepsilon q_{m,n}}{\matrizp{m}{n}{t_{1}}{t_2}}<\left|\sum_{k=0}^{\rb-1}\frac{\sumacoef{m}{n}{t_{1}}{t_2}{k}}{\matrizp{m}{n}{t_{1}}{t_2}}\lambda^{-kp_m}\right|\leq 1.$$
Notice that $\set{1,\lambda^{-p_m},\ldots,\lambda^{-(\rb-1)p_m}}$ is the complete set of $\rb$-th roots of unity if $m$ is large enough, and we
have a convex combination of them. Applying Lemma \ref{lemm:geometrico} we deduce that there exists $\indice{k}{m}{n}{t_{1}}{t_2} \in \{0,\ldots,\rb-1\}$ such that
$$1-\frac{C\varepsilon q_{m,n}}{\matrizp{m}{n}{t_{1}}{t_2}}<\frac{\sumacoef{m}{n}{t_{1}}{t_2}{\indice{k}{m}{n}{t_{1}}{t_2}}}{\matrizp{m}{n}{t_{1}}{t_2}}\leq 1,$$ or equivalently,
$$0\leq\frac{\matrizp{m}{n}{t_{1}}{t_2}}{q_{m,n}}-\frac{\sumacoef{m}{n}{t_{1}}{t_2}{\indice{k}{m}{n}{t_{1}}{t_2}}}{q_{m,n}}<C\varepsilon.$$

The constructed sequence depends on $\varepsilon$. Taking a sequence $(\varepsilon_{\ell};\ell\in \NN)$ tending to zero and using a diagonal process one obtains the desired sequence 
$$(\indice{k}{m}{n}{t_{1}}{t_2}; m,n\in \NN, m<n).$$ 

\smallskip

(3)$\implica$(1). Fix $\varepsilon>0$. There exists a positive integer $N$ large enough such that for any 
$n>m>N$ and $t_{1}\in \{1,\ldots,d\}$,
\begin{equation}\label{hypothesisthree}
0\leq\frac{\matrizp{m}{n}{t_{1}}{t_2}}{q_{m,n}}-\frac{\sumacoef{m}{n}{t_{1}}{t_2}{\indice{k}{m}{n}{t_{1}}{t_2}}}{q_{m,n}}=\sum_{\substack{k=0 \\ k\neq \indice{k}{m}{n}{t_{1}}{t_2}}}^{\rb-1}\frac{\sumacoef{m}{n}{t_{1}}{t_2}{k}}{q_{m,n}}<\varepsilon.
\end{equation}

So, using relations \eqref{eq:suma_en_sumacoef} and \eqref{hypothesisthree} we deduce that 

\begin{align*}
\abs{\suma{m}{n}{t_{1}}{t_2}}&=
\abs{
\sumacoef{m}{n}{t_1}{t_2}{\indice{k}{m}{n}{t_1}{t_2}} \lambda^{-p_{m}\indice{k}{m}{n}{t_1}{t_2}}+
\sum_{\substack{k=0 \\ k\neq {\indice{k}{m}{n}{t_1}{t_2}}}}^{\rb-1} \sumacoef{m}{n}{t_{1}}{t_2}{k} \lambda^{-p_{m}k}
}\\
&\geq \sumacoef{m}{n}{t_1}{t_2}{\indice{k}{m}{n}{t_1}{t_2}} - \sum_{\substack{k=0 \\ k\neq {\indice{k}{m}{n}{t_1}{t_2}}}}^{\rb-1} \sumacoef{m}{n}{t_{1}}{t_2}{k}\\
&\geq \sumacoef{m}{n}{t_1}{t_2}{\indice{k}{m}{n}{t_1}{t_2}} - \epsilon q_{m,n}.
\end{align*}
From this inequality,  \eqref{eq:desigualdad_suma_matriz} and \eqref{hypothesisthree} we get
$$ \frac{\sumacoef{m}{n}{t_{1}}{t_2}{\indice{k}{m}{n}{t_{1}}{t_2}}}{q_{m,n}}-\varepsilon \leq
 \frac{\abs{\suma{m}{n}{t_{1}}{t_2}}}{q_{m,n}} \leq \frac{\sumacoef{m}{n}{t_{1}}{t_2}{\indice{k}{m}{n}{t_{1}}{t_2}}}{q_{m,n}}+\varepsilon.$$

Finally, from \eqref{eq:matriz_en_sumacoef} and \eqref{hypothesisthree} applied to these last inequalities we deduce that
$$\frac{\matrizp{m}{n}{t_{1}}{t_2}}{q_{m,n}}-2\varepsilon \leq \frac{\abs{\suma{m}{n}{t_{1}}{t_2}}}{q_{m,n}} \leq
\frac{\matrizp{m}{n}{t_{1}}{t_2}}{q_{m,n}}+\varepsilon.$$

Adding over $t_{1} \in V_{m}$ we get
$$\abs{\sum_{t_{1}\in V_m}\frac{\abs{\suma{m}{n}{t_{1}}{t_2}}}{q_{m,n}}-1}\leq 2d\varepsilon.$$
Property (1) follows since this inequality is valid for any $n>m>N$ given $\varepsilon>0$.   
\end{proof}

Notice that the sequence $(\indice{k}{m}{n}{t_{1}}{t_2}; m,n \in \NN, m<n)$ in statement (3) of Lemma 
\ref{lemm:suma_matriz_coef_equivalentes} is not necessarily uniquely defined.

\subsection{Constructing a partition from Theorem \ref{theo:principal}}
The next lemma allows to construct several partitions of the vertices in a level of the Bratteli diagram such that
the local order of most of the arrows starting in a vertex of an atom of such partition ending in the same vertex of a further level must be congruent modulo $\rb$. This is crucial to get Corollary \ref{cor:condwithpartition}.

\begin{lemm}\label{lemm:indice_nivelmenor_sobre}
For $t_{2} \in \{1,\ldots,d\}$ asssume that any of the equivalent conditions in Lemma \ref{lemm:suma_matriz_coef_equivalentes} holds.
For each $t_1 \in \{1,\ldots,d\}$ fix a sequence  
$(\indice{k}{m}{n}{t_{1}}{t_2};$ $ m,n \in \NN, m<n)$ as in statement (3) of Lemma \ref{lemm:suma_matriz_coef_equivalentes}. Consider the map 
\begin{eqnarray*}
   \Psi_{m,n,t_2}: \{1,\ldots,d\} &\to& \set{0,\ldots,\rb-1} \ . \nonumber\\
   t_{1} &\mapsto& \indice{k}{m}{n}{t_{1}}{t_2} \label{eq:definicion_funcion_phi}
\end{eqnarray*}
Then, 
\begin{enumerate}
\item for any large enough $m,n \in \NN$ with $m<n$, $\Psi_{m,n,t_2}$ is onto,
\item for any $k\in \{0,\ldots,\rb-1\}$, 
$$\sum_{t_{1}\in \Psi_{m,n,t_2}^{-1}(k)}\frac{\sumacoef{m}{n}{t_{1}}{t_2}{k}}{q_{m,n}}\xrightarrow[m,n\to\infty]{}\frac{1}{\rb}$$ 
uniformly in $m,n\in \NN$ with $m<n$. 
\end{enumerate}
\end{lemm}
\begin{proof}
(1) Fix $0<\varepsilon < 1/(d+1)^{2}$. For any $t_{1}\in \{1,\ldots,d\}$ we have
$$\frac{\matrizp{m}{n}{t_1}{t_2}}{q_{m,n}}-\frac{\sumacoef{m}{n}{t_1}{t_2}{\indice{k}{m}{n}{t_1}{t_2}}}{q_{m,n}}=\sum_{\begin{subarray}{c}k=1\\ k\neq \indice{k}{m}{n}{t_{1}}{t_2}\end{subarray}}^{\rb-1}\frac{\sumacoef{m}{n}{t_1}{t_2}{k}}{q_{m,n}}.$$
Then, since by hypothesis $t_{2}$ verifies condition (3) of Lemma \ref{lemm:suma_matriz_coef_equivalentes}, for any $m,n \in \NN$ with $m<n$ large enough 
$\frac{\sumacoef{m}{n}{t_{1}}{t_2}{k}}{q_{m,n}}<\varepsilon$ for all $t_{1} \in \{1,\ldots,d\}$ and
$k \neq\indice{k}{m}{n}{t_{1}}{t_2}$. Since $q_{n}$ goes to infinity with $n$, then considering larger values of $m,n$ we can also assume that $1/q_{m,n}<\varepsilon$. 

If assertion (1) of the lemma is not true, then for some large $m,n$ with $m<n$, there is 
$k \in \{0,\ldots,\rb-1\}\setminus \operatorname{Im}\Psi_{m,n,t_2}$. Hence, by the previous considerations and equality \eqref{eq:sumacoef_torres_todas} 
$$\frac{1}{\rb}-\varepsilon < \sum_{t_{1}\in V_{m}}\frac{\sumacoef{m}{n}{t_{1}}{t_2}{k}}{q_{m,n}}<d\varepsilon,$$ 
which, by the choice of $\varepsilon$, contradicts the fact that $\rb\leq d$.
\medskip

(2) Fix $\varepsilon>0$. By part (1), there exists $N \in \NN$ such that for all $n>m>N$, 
$\Psi_{m,n,t_2}$ is surjective. Taking a larger $N$ if necessary we can also assume that $1/q_{m,n}$ and 
$\frac{\sumacoef{m}{n}{t_{1}}{t_2}{k}}{q_{m,n}}$ are less than $\varepsilon$ for all $t_{1}\in \{1,\ldots,d\}$ and $k\neq\indice{k}{m}{n}{t_{1}}{t_2}$.

Let $k$ be an element in  $\{0,\ldots,\rb-1\}$. By \eqref{eq:sumacoef_torres_todas} the following inequalities hold for all $n>m>N$:
\begin{align*}
  \ds\frac{1}{\rb}-\varepsilon & <  \ds\sum_{t_{1}\in V_{m}}\frac{\sumacoef{m}{n}{t_{1}}{t_2}{k}}{q_{m,n}} \leq  \ds\frac{1}{\rb}+\varepsilon, \\
  \ds\frac{1}{\rb}-\varepsilon & < \ds\sum_{t_{1}\in \Psi_{m,n,t_2}^{-1}(k)}\frac{\sumacoef{m}{n}{t_{1}}{t_2}{k}}{q_{m,n}}+\sum_{t_{1}\notin \Psi_{m,n,t_2}^{-1}(k)}\frac{\sumacoef{m}{n}{t_{1}}{t_2}{k}}{q_{m,n}}  \leq  \ds\frac{1}{\rb}+\varepsilon, \\
 \ds\frac{1}{\rb}-d\varepsilon & <  \ds\sum_{t_{1}\in \Psi_{m,n,t_2}^{-1}(k)}\frac{\sumacoef{m}{n}{t_{1}}{t_2}{k}}{q_{m,n}} \leq \ds\frac{1}{\rb}+\varepsilon.
\end{align*}
We have proved that
$$\sum_{t_{1}\in \Psi_{m,n,t_2}^{-1}(k)}\frac{\sumacoef{m}{n}{t_{1}}{t_2}{k}}{q_{m,n}}\xrightarrow[m,n\to\infty]{}\frac{1}{\rb}$$ uniformly in $m,n\in \NN$ with $m<n$, 
which ends the proof. 
\end{proof}

From the proof of the previous lemma one can deduce that the values of $\indice{k}{m}{n}{t_{1}}{t_2}$ are ultimately uniquely defined 
if $\displaystyle \liminf_{m,n \to \infty, m<n}\frac{\matrizp{m}{n}{t_1}{t_2}}{q_{m,n}} >0$. 

\section{Proof of Theorem \ref{theo:principal}}
\label{mainproofs}

In all this section $(X,T)$, $\mu$ and $(q_{n};n\geq 1)$ are set as in Theorem \ref{theo:principal}. 

\subsection{Proof that the technical condition is necessary.}
It is enough to consider a non continuous eigenvalue $\lambda=\exp(2i \pi /b)$ of $(X,T)$ for the ergodic measure $\mu$. Let $f\in L^{2}(X,\mu)$ be an associated eigenfunction with $|f|=1$. 

\begin{proof}[Proof that the technical condition is necessary.]
Recall that $b/(b,p_n)$  is equal to $\rb$ for all $n$ large enough. We know from Theorem \ref{theo:vpsonracionales} that $2 \leq \rb \leq d$. Otherwise, if $\rb=1$ the system would be linearly recurrent and $\lambda$ a continuous eigenvalue, as was discussed before stating Theorem \ref{theo:principal}. Thus we only need to prove statement (2) of the theorem. 
\medskip

It is enough to prove that for all $t_1 \in \{1,\ldots,d\}$ and $t_2\in I_{\mu}$,
\begin{equation}
\label{eq:PvsSum}
\frac{\matrizp{m}{n}{t_1}{t_2}}{q_{m,n}}-\frac{\abs{\sum_{s\in\sufijocbb{m}{n}{t_1}{t_2}}\lambda^{-p_ms}}}{q_{m,n}}\xrightarrow[m,n\to\infty]{} 0 
\end{equation}
uniformly in $m,n\in \NN$ with $m<n$. From here, we finish the proof adding over $t_{1} \in \{1,\ldots, d\}$.
\medskip

First, we integrate $f$ over $B_{m}(t_{1})$ and use the decomposition given in \eqref{eq:base_en_bases_superiores}: 
\begin{align*}
\int_{B_{m}(t_{1})} f d\mu 
&= \sum_{t_{2}\in V_{n}}\sum_{s\in\sufijocbb{m}{n}{t_{1}}{t_{2}}}\int_{T^{-p_ms}B_{n}(t_{2})}f d\mu \\
&= \sum_{t_{2}\in V_{n}}\sum_{s\in\sufijocbb{m}{n}{t_{1}}{t_{2}}}\int_{B_{n}(t_{2})}f\circ T^{-p_ms} d\mu \\
&= \sum_{t_{2}\in V_{n}}\left ( \sum_{s\in\sufijocbb{m}{n}{t_{1}}{t_{2}}} \lambda^{-p_ms} \right ) \int_{B_{n}(t_{2})}f d\mu . \\
\end{align*}
But, from \eqref{eq:def_c_y_ro}, we have that 
$$\int_{B_m(t_{1})}f \, d\mu=\mu_m(t_{1})c_m(t_{1})\lambda^{-\rho_m(t_{1})}, \ \int_{B_n(t_{2})}f \, d\mu=\mu_n(t_{2})c_n(t_{2})\lambda^{-\rho_n(t_{2})}.$$

Thus, substituting the corresponding expressions in the previous deduction we get
\begin{align*}
\mu_m(t_{1})c_m(t_{1})\lambda^{-\rho_m(t_{1})}&= \sum_{t_{2}\in V_{n}}\left ( \sum_{s\in\sufijocbb{m}{n}{t_{1}}{t_{2}}} \lambda^{-p_ms} \right ) \mu_n(t_{2})c_n(t_{2})\lambda^{-\rho_n(t_{2})} \\
\mu(\tau_{m}=t_{1})c_m(t_{1})\lambda^{-\rho_m(t_{1})}&= 
\sum_{t_{2}\in V_{n}} 
\frac{\sum_{s\in\sufijocbb{m}{n}{t_{1}}{t_{2}}} \lambda^{-p_ms}}{q_{m,n}}  
\mu(\tau_{n}=t_{2})c_n(t_{2})\lambda^{-\rho_n(t_{2})}, \\
\end{align*}
where in the last equality we have used the relations 
$\mu(\tau_{m}=t_{1})=p_{m} \mu_{m}(t_{1})$, $\mu(\tau_{n}=t_{2})=p_{n} \mu_{n}(t_{2})$ and $p_{n}/p_{m}=q_{m,n}$.
Using \eqref{eq:def_suma} we get the expression
\begin{equation}\label{eq:torren_en_torresm_con_suma}
   \mu(\torref{m}=t_1)c_m(t_1)\lambda^{-\rho_m(t_1)}=\sum_{t_{2}\in V_n}\frac{\suma{m}{n}{t_1}{t_{2}}}{q_{m,n}}\mu(\torref{n}=t_{2})c_n(t_{2})\lambda^{-\rho_n(t_{2})}.
\end{equation}

From \eqref{eq:measure} we have that for $0<m<n$ and $t_1\in \{1,\ldots,d\}$
\begin{equation}\label{eq:torren_en_torresm}
   \mu(\torref{m}=t_1)=\sum_{t_{2}\in V_n}\frac{\matrizp{m}{n}{t_1}{t_{2}}}{q_{m,n}}\mu(\torref{n}=t_{2}).
\end{equation} 

Then, taking absolute value in \eqref{eq:torren_en_torresm_con_suma} and using \eqref{eq:desigualdad_suma_matriz} and \eqref{eq:torren_en_torresm} we deduce
\begin{eqnarray*}
  \mu(\torref{m}=t_1)c_m(t_1)&\leq & \sum_{t_{2}\in V_n}\frac{\abs{\suma{m}{n}{t_1}{t_{2}}}}{q_{m,n}}\mu(\torref{n}=t_{2})c_n(t_{2}) \\
  &\leq & \sum_{t_{2}\in V_n}\frac{\abs{\suma{m}{n}{t_1}{t_{2}}}}{q_{m,n}}\mu(\torref{n}=t_{2}) \\
  &\leq & \sum_{t_{2}\in V_n}\frac{\matrizp{m}{n}{t_1}{t_{2}}}{q_{m,n}}\mu(\torref{n}=t_{2}) \\
  &=& \mu(\torref{m}=t_1).
\end{eqnarray*}
Notice that in the second inequality we have used that $c_n(t_{2})\leq 1$ for any $n\in\NN$ and $t_{2}\in V_n$. 

Finally, applying Lemma \ref{lemm:c_tiende_a_uno} in the preceding inequalities we deduce that
$$\sum_{t_{2}\in V_n}\left(\frac{\matrizp{m}{n}{t_1}{t_{2}}}{q_{m,n}}-\frac{\abs{\suma{m}{n}{t_1}{t_{2}}}}{q_{m,n}}\right)\mu(\torref{n}=t_{2})\xrightarrow[m,n\to \infty]{} 0$$ uniformly in $m,n\in \NN$ with $m< n$. 
If $t_{2} \in I_{\mu}$, then $\mu(\torref{n}=t_{2})> \delta$ (recall $\delta$ comes from the cleanliness property of the diagram). Therefore, the desired convergence in \eqref{eq:PvsSum} holds. 
\end{proof}

\subsection{Proof that the technical condition is sufficient.}

For this proof we will need the following result from \cite{necesariasuficiente} that we adapt to the language of Bratteli-Vershik systems.

\begin{theo}
\label{theorem:vpmeasurable}
Let $(X,T)$ be a minimal Cantor system given by a proper Bratteli-Vershik system.
A complex number $\lambda$ is an eigenvalue of
$(X,T)$ with respect to the ergodic probability measure $\mu$ if and only if
there exists a sequence of real functions
$\rho_n : V_{n} \rightarrow \RR$, $n\in \NN$, such that
\begin{equation}
\label{vp}
\lambda^{r_n (x) +  \rho_n (\tau_n (x))}
\hbox{ converges } \end{equation}
for $\mu$-almost every $x\in X$ when  $n$ tends to infinity.
\end{theo}
We recall that $r_n(x)= \sufijofb{0}(x) + \sum_{i=1}^{n-1}p_i\sufijofb{i}(x)$ is the entrance time of $x$ to 
$B_{n}(\tau_{n}(x))$ (see \eqref{eq:retorno_toeplitz}). 

\begin{proof}[Proof that the technical condition is sufficient.] We notice that condition (2) in Theorem \ref{theo:principal} is stable under telescoping, so we will telescope our 
Bratteli-Vershik representation freely.

\subsubsection{Constructing a partition.}
Take $t_{2} \in I_{\mu}$ and $m,n \in \NN$ with $m<n$ enough large. Notice that our 
hypothesis is condition (1) in Lemma \ref{lemm:suma_matriz_coef_equivalentes} with $t_{2} \in I_{\mu}$. 
Thus, for any $t_{1} \in \{1,\ldots,d\}$ there exist $k_{m,n}(t_{1},t_{2})$ given by condition (3) of Lemma \ref{lemm:suma_matriz_coef_equivalentes} and the map 
$\Psi_{m,n,t_2}: \{1,\ldots,d\} \to \{0,\ldots,\rb-1\}$ given by Lemma \ref{lemm:indice_nivelmenor_sobre}.
Define 
$$
\H_{m,n,t_2}=\set{A_{m,n,t_2}^{(0)}, A_{m,n,t_2}^{(1)}, \ldots, A_{m,n,t_2}^{(\rb-1)}},
$$ 
where
$A_{m,n,t_2}^{(k)}=\Psi_{m,n,t_2}^{-1}(k)$ for $k \in \{0,\ldots,\rb-1\}$.

From Lemma \ref{lemma:contraction-partition} we can suppose after telescoping that
$\H_{m,n,t_{2}} = \H_{m,m+1,t_{2}}$
for all $m,n \in \NN$ with $m<n$ and $t_{2}\in I_{\mu}$. 
Thus we set $\H_{m,n,t_{2}} = \H_{m,t_{2}}$ and $A_{m,n,t_2}^{(k)}= A_{m,t_2}^{(k)}$ for $k\in \{0,\ldots,\rb-1\}$. In addition, after another telescoping, 
we can suppose $A_{m,t_{2}}^{(k)} = A_{m',t_{2}}^{(k)}$ for all $m,m' \geq 1$ and $k\in \{0,\ldots,\rb-1\}$. We set $\H_{t_{2}}=\H_{m,t_{2}}$, $A_{t_{2}}^{(k)} = A_{m,t_{2}}^{(k)}$ and thus
$k(t_1,t_2)=k_{m,n}(t_1,t_2)$ for any $m,n\in \NN$ with $m<n$, $t_{1} \in \{1,\ldots,d\}$  and $t_{2} \in I_{\mu}$.
\medskip

\subsubsection{Constructing a good set of full measure.}
For $m, n\in \NN$ with $m<n$ consider the set
$$
  \mathcal{C}_{m,n}=\set{\tau_{n} \in I_{\mu}, \ \sufijofbb{m}{n}\not\equiv_\rb k(\tau_m, \tau_{n})}
  \cup \{\tau_{n}\not	 \in I_{\mu}\}.
$$ 
Recall that the map $k(t_{1},t_{2})$ has been defined only for $t_{2}\in I_{\mu}$. 
Let us compute the measure of $\mathcal{C}_{m,n}$:

\begin{align*}
    &\mu(\mathcal{C}_{m,n}) \\
    &= \sum_{t_2\in I_{\mu}}\sum_{t_1\in V_m}\mu(\tau_m=t_1, \tau_{n}=t_2, \sufijofbb{m}{n}\not\equiv_\rb k(t_1,t_2)) + \mu(\tau_{n}\not	 \in I_{\mu}) \\
    &= \sum_{t_2\in I_{\mu}}\sum_{t_1\in V_m}\left(\matrizp{m}{n}{t_{1}}{t_{2}}-\sumacoef{m}{n}{t_1}{t_2}{k(t_1 ,t_2)}\right) p_m \ \mu_{n}(t_2) + \mu(\tau_{n}\not \in I_{\mu}) \\
    &= \sum_{t_2\in I_{\mu}}
    \left(\sum_{t_1\in V_m}\frac{\matrizp{m}{n}{t_{1}}{t_{2}}}{q_{m,n}} -
    \frac{\sumacoef{m}{n}{t_1}{t_2}{k(t_1,t_2)}}{q_{m,n}}\right)\mu(\torref{n}=t_2) + 
    \mu(\tau_{n}\not	 \in I_{\mu}), 
\end{align*}
where we have used that $q_{m,n} \ p_{m}=p_{n}$ and $\mu(\tau_{n}=t_{2})=p_{n} \ \mu_{n}(t_{2})$.
\smallskip

Since condition (3) of Lemma \ref{lemm:suma_matriz_coef_equivalentes} holds for $t_{2} \in I_{\mu}$  and 
$\mu(\tau_{n}\not \in I_{\mu})$ goes to $0$ when $n$ tends to $\infty$ (recall the diagram is clean), then  
$\mu(\mathcal{C}_{m,n})\xrightarrow[m,n\to\infty]{} 0$ uniformly in $m,n \in \NN$ with $m<n$.

Thus, we can telescope the diagram in order that 
\begin{align}
\sum_{n\in \NN} \mu(\mathcal{C}_{n,n+1}) \hbox{ converges}. 
\end{align}
Hence, from the Borel-Cantelli Lemma we deduce that $\mu (\mathcal{C}) = 1$, where 
$$\mathcal{C} = \liminf_{n\to 	\infty} \mathcal{C}_{n,n+1}^{c}=\cup_{N\in \NN} \cap_{n\geq  N} \{\tau_{n}\in I_{\mu}, \sufijofb{n}\equiv_\rb k(\tau_n, \tau_{n+1})\} \ .$$

\subsubsection{Constructing an eigenfunction.}
After telescoping we can suppose that $p_n \equiv_b p$ for some $p\in \{ 0,\ldots , b-1\}$ and 
for all $n\geq 1$. This will transform expressions of the form $\lambda^{-p_{n}s}$ below to $\lambda^{-ps}$, which is independent of $n$.
\medskip

For $m,n \in \NN$ with $m<n$, $t_1 \in \{1,\ldots,d\}$ and $t_{2} \in I_{\mu}$ we have
\begin{eqnarray*}
\sum_{s\in\sufijocbb{m}{n}{t_1}{t_2}} \lambda^{-p_ms} 
&=& 
\sum_{\substack{s\in\sufijocbb{m}{n}{t_1}{t_2} \\ s \equiv_{\rb} k(t_{1},t_{2})}}
\lambda^{-pk(t_{1},t_{2})} + \sum_{\substack{s\in\sufijocbb{m}{n}{t_1}{t_2} \\ s \not \equiv_{\rb} k(t_{1},t_{2})}} \lambda^{-ps}
\\
&=&\matrizp{m}{n}{t_{1}}{t_{2}} \lambda^{-pk (t_1,t_2)} + 
\sum_{\substack{s\in\sufijocbb{m}{n}{t_1}{t_2} \\ s \not \equiv_{\rb} k(t_{1},t_{2})}} 
\left(\lambda^{-ps}-\lambda^{-pk(t_{1},t_{2})}\right) ,\\
\end{eqnarray*}
where we have used that $\#\sufijocbb{m}{n}{t_1}{t_2}=\matrizp{m}{n}{t_{1}}{t_{2}}$. 
Also, since
$$\#\{s\in\sufijocbb{m}{n}{t_1}{t_2}\ ; \ s \not \equiv_{\rb} k(t_{1},t_{2})\}=
\matrizp{m}{n}{t_{1}}{t_{2}}-\sumacoef{m}{n}{t_{1}}{t_{2}}{k(t_{1},t_{2})},
$$
we have that 
$$
\left |
\sum_{\substack{s\in\sufijocbb{m}{n}{t_1}{t_2} \\ s \not \equiv_{\rb} k(t_{1},t_{2})}} (\lambda^{-ps}-\lambda^{-pk(t_{1},t_{2})})
\right |
\leq 2 \cdot
(\matrizp{m}{n}{t_{1}}{t_{2}}-\sumacoef{m}{n}{t_{1}}{t_{2}}{k(t_{1},t_{2})})
$$
As mentioned before, condition (2) of the main theorem using $t_{2} \in I_{\mu}$ implies that the equivalent conditions in Lemma \ref{lemm:suma_matriz_coef_equivalentes} hold. So, by Lemma \ref{lemm:suma_matriz_coef_equivalentes} (3), for $t_1 \in \{1,\ldots,d\}$ and $t_{2} \in I_{\mu}$ we have 
$$
\frac{\matrizp{m}{n}{t_{1}}{t_{2}}-\sumacoef{m}{n}{t_{1}}{t_{2}}{k(t_{1},t_{2})}}{q_{m,n}}
\xrightarrow[m,n\to\infty]{} 0
\textrm{\; \; \; \;}$$
uniformly in $m,n \in \NN$ with $m<n$. 

We summarise previous discussion. Fix a real number $\epsilon >0$. 
Then, for all large enough $m,n \in \NN$ with $m<n$, $t_1 \in \{1,\ldots,d\}$ and $t_{2} \in I_{\mu}$
we can write
\begin{equation}
\label{eq:decomposition}
\frac{1}{q_{m,n}}\sum_{s\in\sufijocbb{m}{n}{t_1}{t_2}}\lambda^{-p_ms} = \frac{\matrizp{m}{n}{t_{1}}{t_{2}}}{q_{m,n}} \lambda^{-pk(t_1,t_2)} + \epsilon_{m,n}(t_1,t_2),
\end{equation}
where $\epsilon_{m,n}(t_1,t_2)$ is a complex number with $|\epsilon_{m,n} (t_1,t_2)| \leq \epsilon$. 
\medskip

Now, consider $\ell,m,n \in \NN$ with $\ell<m<n$ enough large (such that the different uses of \eqref{eq:decomposition} below are valid), $t_1 \in \{1,\ldots,d\}$ and $t_{3} \in I_{\mu}$. Then, by using \eqref{eq:interpolacion_sufijos} to get the second equality and \eqref{eq:decomposition} three times, we get  
\begin{align*}
& \frac{\matrizp{\ell}{n}{t_{1}}{t_{3}}}{q_{\ell,n}} \lambda^{-pk(t_1,t_3)} + \epsilon_{\ell,n} (t_1,t_3) \\
= & \frac{1}{q_{\ell,n}}\sum_{s\in\sufijocbb{\ell}{n}{t_1}{t_3}}\lambda^{-p_\ell s} \\
= & \frac{1}{q_{\ell,n}}\sum_{t_{2} \in V_{m}} 
\sum_{s_1\in\sufijocbb{l}{m}{t_{1}}{t_2}}
\sum_{s_2\in \sufijocbb{m}{n}{t_2}{t_{3}} }
\lambda^{-p_\ell s_1 -p_{m} s_2}\\
= & \sum_{t_{2} \in V_{m}} 
\left (\frac{1}{q_{\ell,m}} \sum_{s_1\in\sufijocbb{\ell}{m}{t_{1}}{t_2}} \lambda^{-p_\ell s_1} \right)
\left (\frac{1}{q_{m,n}} \sum_{s_2\in \sufijocbb{m}{n}{t_2}{t_{3}}} \lambda^{-p_{m} s_2} \right )\\
=&
\sum_{t_{2} \in I_{\mu}}
\left(\frac{\matrizp{\ell}{m}{t_{1}}{t_{2}}}{q_{\ell,m}}\lambda^{-pk(t_1,t_{2})} +\epsilon_{\ell,m}(t_1,t_{2})\right) 
\left(\frac{\matrizp{m}{n}{t_{2}}{t_{3}}}{q_{m,n}}\lambda^{-pk(t_{2},t_3)} +\epsilon_{m,n}(t_{2},t_3)\right) \\
&+ \sum_{t_{2} \in V_{m}\setminus I_{\mu}}
\left(\frac{1}{q_{\ell,m}}\sum_{s_1\in\sufijocbb{\ell}{m}{t_{1}}{t_2}} \lambda^{-p_\ell s_1}\right )
\left(\frac{\matrizp{m}{n}{t_{2}}{t_{3}}}{q_{m,n}}\lambda^{-pk(t_{2},t_3)} +\epsilon_{m,n}(t_{2},t_3)\right)\\
\end{align*}

Set $k(t_1,t_{2})=0$ for $t_{1} \in \{1,\ldots,d\}$ and $t_{2}\notin I_{\mu}$ (recall that this map is only defined for $t_{2} \in I_{\mu}$). Adding and subtracting the terms  
$\frac{\matrizp{\ell}{m}{t_{1}}{t_{2}}}{q_{\ell,m}}\lambda^{-pk(t_1,t_{2})}$ when $t_{2} \notin  I_{\mu}$ in the last equality of previous deduction gives

\begin{align*}
& \frac{\matrizp{\ell}{n}{t_{1}}{t_{3}}}{q_{\ell,n}} \lambda^{-pk(t_1,t_3)} + \epsilon_{\ell,n} (t_1,t_3) \\
&=
\sum_{t_{2} \in I_{\mu}}
\left(\frac{\matrizp{\ell}{m}{t_{1}}{t_{2}}}{q_{\ell,m}}\lambda^{-pk(t_1,t_{2})} +\epsilon_{\ell,m}(t_1,t_{2})\right) \\
& \hskip 6cm \cdot
\left(\frac{\matrizp{m}{n}{t_{2}}{t_{3}}}{q_{m,n}}\lambda^{-pk(t_{2},t_3)} +\epsilon_{m,n}(t_{2},t_3)\right) \\
&+
\sum_{t_{2} \in V_{m}\setminus I_{\mu}}
\left(\frac{\matrizp{\ell}{m}{t_{1}}{t_{2}}}{q_{\ell,m}}\lambda^{-pk(t_1,t_{2})} \right)
\left(\frac{\matrizp{m}{n}{t_{2}}{t_{3}}}{q_{m,n}}\lambda^{-pk(t_{2},t_3)} +\epsilon_{m,n}(t_{2},t_3)\right) \\
&+ \sum_{t_{2} \in V_{m}\setminus I_{\mu}}
\left(\frac{1}{q_{\ell,m}}\sum_{s_1\in\sufijocbb{\ell}{m}{t_{1}}{t_2}} \lambda^{-p_\ell s_1}
- \frac{\matrizp{\ell}{m}{t_{1}}{t_{2}}}{q_{\ell,m}}\lambda^{-pk(t_1,t_{2})}\right ) \\
& \hskip 6cm \cdot
\left(\frac{\matrizp{m}{n}{t_{2}}{t_{3}}}{q_{m,n}}\lambda^{-pk(t_{2},t_3)} +\epsilon_{m,n}(t_{2},t_3)\right).\\
\end{align*}
Finally, multiplying the terms we get that 
\begin{align}
\label{eq:maincomputation}
\frac{\matrizp{\ell}{n}{t_{1}}{t_{3}}}{q_{\ell,n}} \lambda^{-pk(t_1,t_3)} &+ \epsilon_{\ell,n} (t_1,t_3) \nonumber \\
&=\epsilon' +
\sum_{t_{2} \in V_{m}}
\frac{\matrizp{\ell}{m}{t_{1}}{t_{2}} \matrizp{m}{n}{t_{2}}{t_{3}}}{q_{\ell,n}} 
\lambda^{-p(k(t_1,t_{2})+k(t_{2},t_3))}, \nonumber \\
\end{align}
where 
\begin{align}
\label{epsilonprime}
|\epsilon'|\leq 
2d\epsilon + d\epsilon^{2}+ d \epsilon +
\sum_{t_{2} \in V_{m}\setminus I_\mu} 2 \cdot \frac{\matrizp{m}{n}{t_{2}}{t_{3}}}{q_{m,n}}
+ 2 d\epsilon.
\end{align}
But, for $t_{2} \not\in I_\mu$, $t_{3} \in I_{\mu}$ and any large $m,n \in \NN$ with $m<n$ we have that 
$\mu ( \tau_n = t_3 ) \geq \delta$ and  $\mu ( \tau_m = t_{2} ) \leq \delta\epsilon$, where $\delta$ comes from the definition of a clean Bratteli-Vershik representation. Consequently, using equality \eqref{eq:torren_en_torresm}, 
we have that
\begin{align}
\label{eq:outofImu}
\frac{\matrizp{m}{n}{t_{2}}{t_{3}}}{q_{m,n}} \leq \frac{\mu  ( \tau_m = t_{2} )}{\mu  ( \tau_n = t_3 )} 
\leq \frac{\mu  ( \tau_m = t_{2} )}{ \delta } \leq \epsilon.
\end{align}
Thus, combining \eqref{eq:outofImu} in \eqref{epsilonprime}, we get 
$$|\epsilon'| \leq 5 d\epsilon + 2 d \epsilon \leq 8 d\epsilon.$$

Now, a simple reordering of terms in \eqref{eq:maincomputation} gives 
\begin{align}
\label{secondcomputation}
 1 + \frac{q_{\ell,n}}{\matrizp{\ell}{n}{t_{1}}{t_{3}}}&(\epsilon_{\ell,n} (t_1,t_3)
-  \epsilon') \lambda^{pk(t_{1},t_{3})}\\
& =\sum_{t_{2} \in V_{m}}
\frac{\matrizp{\ell}{m}{t_{1}}{t_{2}}
\matrizp{m}{n}{t_{2}}{t_{3}}}{\matrizp{\ell}{n}{t_{1}}{t_{3}}} 
\lambda^{p(k(t_1,t_3)-k(t_1,t_{2})-k(t_{2},t_3))}.
\end{align}

Recall from Lemma~\ref{lemma:low-independence} that for every $\ell \in \NN$ enough large there exist integers $m,n$ with $n>m>\ell$ such that for all $t_1,t_{2},t_3\in I_{\mu}$, 
\begin{align}
\label{eq:low-independance}
\frac{\matrizp{\ell}{n}{t_{1}}{t_{3}}}{q_{\ell,n}} \geq \frac{\delta}{3}, \
\frac{\matrizp{\ell}{m}{t_{1}}{t_{2}}}{q_{\ell,m}} \geq \frac{\delta}{3} \text{ and } 
\frac{\matrizp{m}{n}{t_{2}}{t_{3}}}{q_{m,n}} \geq \frac{\delta}{3}.
\end{align}

Then, if considering $t_{1},t_{3} \in I_{\mu}$ and fixing integers $\ell,m,n\in\NN$ with $\ell<m<n$ enough large to verify \eqref{eq:low-independance}, and using \eqref{secondcomputation}, we get

\begin{align}
\label{eq:convexsum}
1 + \epsilon''
& = 
\sum_{t_{2}\in V_{m}}
\frac{\matrizp{\ell}{m}{t_{1}}{t_{2}} \matrizp{m}{n}{t_{2}}{t_{3}}}{\matrizp{\ell}{n}{t_{1}}{t_{3}}} 
\lambda^{p(k(t_1,t_3)-k(t_1,t_{2})-k (t_{2},t_3))},
\end{align}
where $|\epsilon''|\leq \hat C \epsilon $ and $\hat C$ is a positive constant only depending on the system.
\medskip

Let us show that $pk(t_1,t_3) \equiv_{b} p(k (t_1,t_{2})+k (t_{2},t_3))$ for all $t_{2}\in I_{\mu}$.
We rewrite the right hand side of \eqref{eq:convexsum}, which is a convex sum, as 
$\sum_{i=0}^{b-1} \alpha_{i}\lambda^i$, where 
$$
\alpha_{i}=\sum_{\{t_{2} \in V_{m}\ ; \ p(k (t_1,t_3) -k (t_1,t_{2})-k (t_{2},t_3)) \equiv_b i \}}
\frac{\matrizp{\ell}{m}{t_{1}}{t_{2}} \matrizp{m}{n}{t_{2}}{t_{3}}}{\matrizp{\ell}{n}{t_{1}}{t_{3}}}. 
$$

By \eqref{eq:convexsum}, we can use Lemma \ref{lemm:geometrico}. Then, there is $i_{0} \in \{0,\ldots, b-1\}$ such that 
$\alpha_{i_{0}} > 1- C \hat C\epsilon$ ($C$ is the constant of Lemma \ref{lemm:geometrico} for the $b$-th roots of unity).  Moreover, if $\epsilon$ was taken small enough, we have that $i_{0}=0$ since the convex combination is close to $1$. But, again using \eqref{eq:low-independance}, for all $t_{2}\in I_{\mu}$,

\begin{align*}
\frac{\matrizp{\ell}{m}{t_{1}}{t_{2}} \matrizp{m}{n}{t_{2}}{t_{3}}}{\matrizp{\ell}{n}{t_{1}}{t_{3}}}= 
\frac{\matrizp{\ell}{m}{t_{1}}{t_{2}}}{q_{\ell,m}} \frac{\matrizp{m}{n}{t_{2}}{t_{3}}}{q_{m,n}}
\frac{q_{\ell,n}}{\matrizp{\ell}{n}{t_{1}}{t_{3}}}
\geq  \frac{\delta^2}{9}>  C \hat C \epsilon
\end{align*}
if $\epsilon$ was taken small enough. 
Since $\alpha_{i_{0}} > 1- C \hat C\epsilon$, then for all $t_{2}\in I_{\mu}$, 
$$p(k(t_1,t_3) -k (t_1,t_{2})-k (t_{2},t_3))\equiv_b i_0=0 \ .$$ 
This proves our claim.
\medskip

Summarising, we have proved that for all $t_{1},t_{2}, t_{3}\in I_{\mu}$,
\begin{equation}\label{eq:cocyle1}
p \cdot k (t_{1},t_{3}) \equiv_{b}  p \cdot k (t_{1},t_{2}) + p \cdot k (t_{2},t_{3}), 
\end{equation}
\begin{equation}\label{eq:cocyle2}
p \cdot k (t_{1},t_{1}) \equiv_{b} 0, \ \  p \cdot k (t_{1},t_{2}) \equiv_{b} -p \cdot k (t_{2},t_{1}).
\end{equation}
\medskip

To finish we will verify the criterium of Theorem \ref{theorem:vpmeasurable} for $\lambda=\exp(2i\pi/b)$. 
Fix an element $t_{0} \in I_{\mu}$ and for each $n\geq 1$ define $\rho_n: V_{n} \to \RR$ by 
$\rho_{n}(t)= - p  k (t_{0},t)$.

Let $x$ be an element in $\mathcal{C}$.
By definition of ${\mathcal C}$, there exists $N \in \NN$ such that for any $n\geq N$, $\tau_{n}(x) \in I_{\mu}$ and 
$\overline{s}_{n} (x) \equiv_\rb   k (\tau_n (x) , \tau_{n+1} (x))$. Notice that, since $p\rb$ is divisible by $b$ (recall that $\rb = p/(b,p)$), then after multiplying by $p$ we get that  
$p\overline{s}_{n} (x) \equiv_b  p k (\tau_n (x) , \tau_{n+1} (x))$. 
Then, for $n \geq N$ one has, 
\begin{align*}
| \lambda^{r_{n+1} (x) + \rho_{n+1} (\tau_{n+1}(x))} -  \lambda^{r_{n} (x) + \rho_n (\tau_{n}(x))}|= &  
| \lambda^{r_{n+1} (x) -r_{n}(x)+ \rho_{n+1} (\tau_{n+1}(x))-\rho_{n}(\tau_{n}(x))} -  1|\\
=& | \lambda^{ p \overline{s}_{n} (x) - pk (t_{0},\tau_{n+1} (x)) + pk (t_{0},\tau_n (x)) } - 1 |\\
=& | \lambda^{ p \overline{s}_{n} (x) - (p k (\tau_n (x),t_{0})+pk (t_{0},\tau_{n+1} (x)))  } - 1 |\\
=& | \lambda^{ p \overline{s}_{n} (x) - pk (\tau_n (x) , \tau_{n+1} (x)) } - 1 |=0,
\end{align*}
where to deduce the second equality we have used \eqref{eq:retorno_toeplitz} and to derive the last one we applied \eqref{eq:cocyle1} and \eqref{eq:cocyle2}. This proves that $\lambda^{r_{n}(x) + \rho_n (\tau_{n}(x))}$ 
is eventually constant, so it converges. We finish the proof using Theorem \ref{theorem:vpmeasurable}. 
\end{proof}

Let us remark that from the previous proof Corollary \ref{coro:cocycle} follows directly. In fact, it is just a reformulation of the last part of the proof.  

\subsection{Proof of Corollary \ref{cor:alternative}}
(1) Let $\mu$ be an ergodic measure such that $\rB_{\mu}$ is non empty. Let $\lambda=\exp(2i\pi a/b)$ be a 
non continuous eigenvalue for $\mu$ such that $b/(b,p_{n})=\rb \in \rB_{\mu}$ for all large enough integers $n\in \NN$.
The hypotheses of Lemma \ref{lemm:indice_nivelmenor_sobre} hold for all $t_{2} \in I_{\mu}$ using this value of $\lambda$. Then, from Lemma \ref{lemm:indice_nivelmenor_sobre} (2), for every $t_2\in I_{\mu}$ and $k\in \{0,\ldots,\rb-1\}$ the sum
\begin{align*}
&\sum_{t_{1}\in \Psi_{m,n,t_2}^{-1}(k)}\frac{\sumacoef{m}{n}{t_{1}}{t_2}{k}}{q_{m,n}} \\
&=\sum_{t_{1}\in \Psi_{m,n,t_2}^{-1}(k)\cap I_{\mu}}\frac{\sumacoef{m}{n}{t_{1}}{t_2}{k}}{q_{m,n}} + 
\sum_{t_{1}\in \Psi_{m,n,t_2}^{-1}(k)\cap I_{\mu}^{c}}\frac{\sumacoef{m}{n}{t_{1}}{t_2}{k}}{q_{m,n}}
\end{align*}
converges uniformly in $m,n \in \NN$ with $m<n$ to $1/\rb$. But, since
$\frac{\sumacoef{m}{n}{t_{1}}{t_2}{k}}{q_{m,n}} \leq \frac{\matrizp{m}{n}{t_{1}}{t_2}}{q_{m,n}},$
from \eqref{eq:p_tiende_a_cero_fuera_de_I} we deduce that 

$$
\sum_{t_{1}\in \Psi_{m,n,t_2}^{-1}(k)\cap I_{\mu}}\frac{\sumacoef{m}{n}{t_{1}}{t_2}{k}}{q_{m,n}}\xrightarrow[m,n\to\infty]{}\frac{1}{\rb}, \ \  
\sum_{t_{1}\in \Psi_{m,n,t_2}^{-1}(k)\cap I_{\mu}^{c}}\frac{\sumacoef{m}{n}{t_{1}}{t_2}{k}}{q_{m,n}} 
\xrightarrow[m,n\to\infty]{} 0,$$ 
converges uniformly in $m,n \in \NN$ with $m<n$.

We deduce that for any large enough $m,n \in \NN$ with $m< n$, $t_2\in I_{\mu}$ and $k\in \{0,\ldots,\rb-1\}$ each set 
$\Psi_{m,n,t_2}^{-1}(k)$ must contain an element of $I_{\mu}$. Thus $\# I_{\mu}\geq \rb$.
\medskip

(2) Let us consider $\mu\in\mathcal{M}_{erg}(X,T)$ such that $\rB_{\mu}\neq\emptyset$. 
Let $\exp(2i\pi/b_1)$ and $\exp(2i\pi/b_2)$ be two different non continuous eigenvalues for
$\mu$. Then, by Bezout's identity, $\exp(2i\pi/\operatorname{lcm}(b_1,b_2))$ is also an eigenvalue for $\mu$.
Moreover, it is a non continuous eigenvalue. Indeed, if this fact is not true, then 
for some $n\in \NN$ and $a \in \ZZ$ we have that 
$1/\operatorname{lcm}(b_1,b_2)=a/p_{n}$. This implies that $1/b_{1}=(a\operatorname{lcm}(b_1,b_2)/b_{1})/p_{n}$ which is a contradiction since $\exp(2i\pi/b_1)$ is a non continuous eigenvalue. This proves our claim.
 
Denote $\operatorname{lcm}(b_1,b_2)$ by $b$. Decomposing 
$b_{1}={\mathfrak b}_{1}\cdot {\mathfrak b}_{2} \cdot {\mathfrak b}_{3} \cdot {\mathfrak b}_{4}$ and 
$b_{2}={\mathfrak b}_{3} \cdot {\mathfrak b}_{4} \cdot {\mathfrak b}_{5} \cdot {\mathfrak b}_{6}$, where 
$(b_{1},b_{2})={\mathfrak b}_{3}\cdot {\mathfrak b}_{4}$, $(b_{1},p_{n})={\mathfrak b}_{2}\cdot {\mathfrak b}_{3}$ and $(b_{2},p_{n})={\mathfrak b}_{3} \cdot {\mathfrak b}_{5}$, we get 
the identity
$$
\operatorname{lcm}\left(\frac{b_1}{(b_1,p_n)},\frac{b_2}{(b_2,p_n)}\right)=\frac{b}{(b,p_n)}.
$$
From this identity follows that it is not possible to have more than one divisibility-maximal element in $\rB_{\mu}$.

(3) For different ergodic measures $\mu$ and $\nu$ we have $I_{\mu}\cap
I_{\nu}=\emptyset$ (recall that the Bratteli-Vershik representation is clean). Then,
$$
\sum_{\mu\in \mathcal{M}_{erg}(X,T)} \#I_{\mu}\leq d.
$$
But, (1) implies that $\rb_{\mu}\leq \#I_{\mu}$ for each $\mu \in {\mathcal M}$, so (3) follows.

(4) As in the proof of (3), we use that for different ergodic measures $\mu$ and $\nu$ we have $I_{\mu}\cap
I_{\nu}=\emptyset$. Hence,  
\begin{align*}
\# \mathcal{M} \leq \# \mathcal{M}_{erg}(X,T) &= \sum_{\mu\in\mathcal{M}_{erg}(X,T)}\#I_{\mu} - \sum_{\mu\in\mathcal{M}_{erg}(X,T)}(\#I_{\mu}-1) \\
&\leq d-\sum_{\mu\in {\mathcal M}}(\rb_{\mu}-1),
\end{align*}
where in the inequality we used (1). This proves (4).

\subsection{Proof of Corollary \ref{coro:ue}}
Consider $\lambda = \exp{(2i\pi /b)}$ with $b$ an integer such that $b/(b,p_n)=d$ for all $n$ large enough. 

First we prove the necessary and sufficient condition given by \eqref{eq:coro7}.
If $\lambda$ is a non continuous eigenvalue, then $\rb_{\mu}$ defined in Corollary \ref{cor:alternative} is equal to $d$. 
In addition, since $\rb_{\mu}=d$, the partition of Corollary \ref{cor:condwithpartition} is made of singletons and we get the property \eqref{eq:coro7} for any $t_{2} \in I_{\mu}$. But, using statement (1) of Corollary \ref{cor:alternative}, one deduces that $I_{\mu}=\{1,\ldots,d\}$. Thus property \eqref{eq:coro7} is true for any $t_{2} \in \{1,\ldots,d\}$.
Clearly, property \eqref{eq:coro7} implies that $\lambda$ is a non continuous eigenvalue  by Corollary 
\ref{cor:condwithpartition}. 

Now, assume $\lambda$ is a non continuous eigenvalue. Using Corollary \ref{cor:alternative} (3) one gets that 
$\mathcal{M}_{erg}(X,T)$ has a unique element, so the system is uniquely ergodic. This proves statement \eqref{uergodic}.

Finally we prove statement \eqref {theo:principal1}. Recall that under our hypothesis equivalent conditions of Lemma \ref{lemm:suma_matriz_coef_equivalentes} hold for any $t_{2} \in \{1,\ldots,d\}$. 
Then, from the equality
$$
\left|\mu(\tau_m=t_{1})-\frac{1}{d}\right|=\abs{\sum_{t_{2}\in V_{n}}\left(\frac{\matrizp{m}{n}{t_{1}}{t_{2}}}{q_{m,n}}-\frac{1}{d}\right)\mu(\tau_n=t_{2})},
$$
\eqref{eq:coro7} and Lemma \ref{lemm:suma_matriz_coef_equivalentes} (2) one gets
\begin{equation*}\label{eq:limite_medidas_1sobred}
\lim_{m\to\infty}\mu(\tau_m=t_{1})=\frac{1}{d}.
\end{equation*}
This proves the desired statement.

\section{Examples}
\label{examples}

\subsection{Example 1: A model example}
We start with a basic \emph{model example} that will be used later to illustrate several behaviors of the eigenvalues with 
respect to the ergodic measures. We start with a general framework to construct a family of examples where the Bratteli-Vershik representations are not necessarily proper. Later we modify this family to obtain proper representations. Finally, we prove that in this family of examples all ergodic measures share the same non continuous eigenvalue $\exp(2i\pi/6)$. 

\subsubsection{}
Define the sequence $q_1=1$, $q_2=2\cdot 5^{2}$ and $q_n=5^{2n}$ for $n>2$. 
First, consider the (non necessarily proper) Toeplitz diagram with the characteristic sequence $(q_n; n\in \NN)$
such that $V_n=\{1,2,3,4,5,6,7\}$ for all $n\geq 1$ and the local order of the $q_{n+1}$ arrows arriving at $t \in V_{n+1}$ is given by the following associated sequences of vertices in $V_{n}$:
$$
t \rightarrow v_{n+1}(t) \; \; \textrm{for all} \; 1\leq t\leq 7,
$$
where each $v_{n+1}(t)$ is a fixed word of length $q_{n+1}$ on the alphabet $V_n$ built in the following way:

(1) Set $W_1=\{1,4,7\}$, $W_2=\{2,5\}$ and $W_3=\{3,6\}$.

(2) For $n\geq 2$ the words $v_{n+1}(1), v_{n+1}(4)$ and $v_{n+1}(7)$ begin with an element of $W_1$, followed by an element
of $W_2$ and this is followed by an element of $W_{3}$. Then we restart from $W_{1}$ and so on. 
Because $q_{n+1}\equiv_{3}1$ for $n \geq 2$, all these three words end with an element of
$W_1$. The words $v_{n+1}(2)$ and $v_{n+1}(5)$ follow the same periodic scheme starting with an element of $W_2$, then of $W_{3}$ and so on (and therefore ending with an
element of $W_2$). And finally the words $v_{n+1}(3)$ and $v_{n+1}(6)$ follow the periodic scheme starting in $W_3$.

(3) Level 2 is built in any way.
\medskip

Define $k: \{1,\ldots,7\} \times \{1,\ldots,7\} \to \{0,1,2\}$ by:  
$k(t_{1},t_{2})=j-i \mod 3$ if $t_{1} \in W_{i}$ and $t_{2} \in W_{j}$.
The following two properties are straightforward:
\smallskip

- For $t_1,t_2,t_3 \in \{1,\ldots,7\}$ we have
\begin{equation}
\label{lem:formula1}
k(t_1,t_3)\equiv_{3}k(t_1,t_2)  + k(t_2,t_3).
\end{equation}
- Let $x$ be an infinite sequence in the ordered Bratteli diagram. For $n\geq 2$, 
\begin{equation}
\label{lem:formula2}
\overline{s}_n(x)\equiv_{3}k(\tau_{n}(x),\tau_{n+1}(x)).
\end{equation}

\subsubsection{} Now we modify a little bit the previously defined local orders to get a proper Bratteli-Vershik representation for the system. To produce the new orders we change sequences 
$v_{n+1}(t)$ into $w_{n+1}(t)$ in such a way that: (1) $w_{n+1}(t)=v_{n+1}(t)$, except for at most a fixed number of letters, say $L$, independent of $n$; (2)  $w_{n+1}(t)$ begins and ends with 1; and (3) $w_{n+1}(t)$ contains every element of $\{1,2,3,4,5,6,7\}$ at least once. This diagram is clearly proper and induces a Toeplitz system of finite rank $(X,T)$. 

Consider any invariant measure $\mu$ on the system. We prove that $\exp(2\pi i/6)$ is a non continuous eigenvalue of $(X,T)$ for $\mu$, and this fact is independent of the measure $\mu$ we choose. In order to do that, we verify conditions (1) to (4) of Corollary \ref{coro:cocycle}. 

By construction $p_{n}\equiv_{6} 2$ and $\rb=6/(6,p_{n})=3$ for all $n\geq 2$, so conditions (1) and (2) hold. 
Condition (3) follows directly from \eqref{lem:formula1}. To prove condition (4) we 
need to find a set of full measure where $\overline{s}_n(x)\equiv_{3}k(\tau_{n}(x),\tau_{n+1}(x))$ 
for all large enough $n\in \NN$. 

Similarly as in the proof of Theorem \ref{theo:principal}, for $n\in \NN$ consider
$
\mathcal{C}_{n}=\{x\in X\talque \overline{s}_{n}(x)\not \equiv_{3} k(\tau_{n}(x),\tau_{n+1}(x)) \}.
$
Since \eqref{lem:formula2} holds before modifying the orders 
and those modifications alter no more than $L$ letters per level, we easily check that
\begin{align*}
\mu(\mathcal{C}_{n})&=
\sum_{t_{1}=1}^{7}\sum_{t_{2}=1}^{7}\mu(\tau_{n}=t_{1},\tau_{n+1}=t_{2}, \overline{s}_{n}(x)\not \equiv_{3} k(t_{1},t_{2})) \\
&\leq \sum_{t_{1}=1}^{7}\sum_{t_{2}=1}^{7} L   h_{n}(t_{1}) \mu_{n+1}(t_{2}) \\
&\leq \sum_{t_{1}=1}^{7}\sum_{t_{2}=1}^{7} \frac{L}{q_{n+1}} \mu(\tau_{n+1}=t_{2}) \\
&\leq \frac{7L}{q_{n+1}}.
\end{align*}
So,
$\sum_{n\geq 1}\mu(\mathcal{C}_{n})$ converges.
Hence, from the Borel-Cantelli Lemma we get $\mu(\mathcal{C})=1$, where $\displaystyle \mathcal{C}=\liminf_{n\to \infty} \mathcal{C}_n^{c}$.

\subsection{Example 2: A first particular case of the model example} In this example we precise the construction of Example 1
in order to show that the model example can produce a uniquely ergodic system, where $\exp(2i\pi /6)$ is a non continuous eigenvalue for the unique invariant measure. In addition, this will illustrate that inequalities in Corollary \ref{cor:alternative} can be strict and that Corollary \ref{coro:ue} is not reversible since one can have $\rb_{\mu}<d$ in the uniquely ergodic case. 

First, for $n\geq 3$ define $c_n$ such that $q_n=12c_n+1$ and define words giving the order of the diagram by
\begin{align*}
w_{n+1}(1) &= (123456723756)^{c_{n+1}}1 \\ 
w_{n+1}(2) &= 1(312645372675)^{c_{n+1}-1}(312)^{3}671 \\
w_{n+1}(3) &= 1(123456723756)^{c_{n+1}-1}(123)^{3}451 \\ 
w_{n+1}(4) &= (156423756723)^{c_{n+1}}1 \\
w_{n+1}(5) &= 1(345612375672)^{c_{n+1}-1}(645)^{3}311 \\ 
w_{n+1}(6) &= 1(156423756723)^{c_{n+1}-1}(723)^{3}121 \\
w_{n+1}(7) &= (153426753726)^{c_{n+1}}1 
\end{align*}
It is straightforward that these orders fit the model construction in Example 1. Also, for any invariant measure $\mu$, the system satisfies: $\mu(\tau_{n}=t)\xrightarrow[n\to\infty]{}1/6$ for $t=2,3,5,6,7$ and 
$\mu(\tau_{n}=t)\xrightarrow[n\to\infty]{}1/12$ for $t=1,4$. The proof is a simple computation. For example, 
$$
\frac{c_{n+1}+1}{q_{n+1}}<\mu(\tau_n=1)<\frac{c_{n+1}+4}{q_{n+1}},
$$
and then we use that $c_{n+1}/q_{n+1}\xrightarrow[n\to\infty]{}1/12$. Since $\#I_{\mu}=7$, then we deduce that the
system is uniquely ergodic. Also, since $3$ divides $\rb_{\mu}$, then $\rb_{\mu}< \# I_{\mu}$. 

\subsection{Example 3: A second particular case of the model example} Here we will use the model example to produce a Bratteli-Vershik system having exactly two ergodic measures. Then, for each one $\exp(2i\pi/6)$ is a non continuous eigenvalue. 
Let us take in the model example the following particular choice of $w_{n+1}(t)$ for $t\in \{1,\ldots, 7\}$ 
and $n\geq 2$. First define $c_n$ so that $q_n=3c_n+1$, and then set:
\begin{align*}
w_{n+1}(1) &= (123)^{c_{n+1}-2}4567231 \\
w_{n+1}(2) &= 13(123)^{c_{n+1}-2}45671 \\
w_{n+1}(3) &= 1(123)^{c_{n+1}-2}456721 \\
w_{n+1}(4) &= 146(456)^{c_{n+1}-2}7231 \\
w_{n+1}(5) &= 14(456)^{c_{n+1}-2}12371 \\
w_{n+1}(6) &= 1(456)^{c_{n+1}-2}123761 \\
w_{n+1}(7) &= 156(456)^{c_{n+1}-2}7231
\end{align*}

As was shown in Theorem 3.3 (2) of \cite{bkms}, any ergodic measure is obtained as an extension of a finite measure on a
system defined on a \emph{subdiagram}. A subdiagram is obtained fixing subsets of vertices 
at each level and considering only the paths which go along the vertices in such subsets. The order is defined naturally following the order of the complete diagram. Here we will fix a unique subset of $\{1,\ldots,7\}$ for all levels. 

Consider the subset $A_{1}=\set{1,2,3}$ and construct the associated subdiagram. 
Using the same nomenclature as before, for levels $n\geq 2$ the corresponding subdiagram has the following induced local orders:
\begin{align*}
 1 &\rightarrow (123)^{c_{n+1}-2}231 \\
 2 &\rightarrow 13(123)^{c_{n+1}-2}1 \\
 3 &\rightarrow 1(123)^{c_{n+1}-2}21
\end{align*}
This order determines a proper diagram that is of Toeplitz type and has the characteristic sequence 
$(\overline{q}_n)_{n\in\NN}$, with $\overline{q}_n=q_n-4$ for $n>2$. Analogously to Example 2, 
we can see that the system $(Y,S)$ induced by this diagram is uniquely ergodic. Moreover, the unique invariant measure $\mu$ of this system can be naturally extended to a finite ergodic measure of $(X,T)$. For a deeper discussion of this extension we refer the reader to \cite{bkms} Section 3. Let us call $\widehat{\mu}$ the normalized extension of $\mu$. Then $\widehat{\mu}$ is an ergodic probability measure on $(X,T)$.

Analogously, consider $A_{2}=\set{4,5,6}$. In this case the corresponding subdiagram has the following local orders. For $n\geq  2$,
\begin{align*}
4 &\rightarrow 46(456)^{c_{n+1}-2} \\
5 &\rightarrow 4(456)^{c_{n+1}-2} \\
6 &\rightarrow (456)^{c_{n+1}-2}6
\end{align*}
This diagram has unique maximal and minimal paths, the words have lengths 
$q_{n+1}-5$, $q_{n+1}-6$ and $q_{n+1}-6$ respectively. As before, one proves that the system $(Z,R)$ associated to this diagram is uniquely ergodic and that the unique ergodic measure 
$\nu$ can be extended to a finite ergodic measure of $(X,T)$. We call $\widehat{\nu}$ the normalized extension of $\nu$.

From Theorem 3.3 (4) in \cite{bkms} one deduces that $(X,T)$ has no other ergodic probability measures than 
$\widehat{\mu}$ and $\widehat{\nu}$. Furthermore, one proves by simple computations that the diagram is clean and $I_{\widehat{\mu}}=\set{1,2,3}$ and  $I_{\widehat{\nu}}=\set{4,5,6}$.

\subsection{Example 4: A small variation of the model example} We provide an example of a finite rank Toeplitz system with two ergodic measures. For one there is a non continuous eigenvalue, while for the other all eigenvalues are continuous. We keep the values for $q_{n}$ of Example 1 but 
we consider the following choice of $w_{n+1}(t)$ for $t\in \{1,\ldots, 7\}$ and $n\geq 2$, where
$c_n$ is such that $q_n=12c_n+1$:
\begin{align*}
w_{n+1}(1) &= (123456423156)^{c_{n+1}-1}(123)^{3}7561 \\
w_{n+1}(2) &= 1(312645342615)^{c_{n+1}-1}(312)^{3}671 \\
w_{n+1}(3) &= 1(123456423156)^{c_{n+1}-2}(123)^{3}751 \\
w_{n+1}(4) &= (156423456123)^{c_{n+1}-1}(123)^{3}7561 \\
w_{n+1}(5) &= 1(345612315642)^{c_{n+1}-1}(645)^{3}371 \\
w_{n+1}(6) &= 1(156423456123)^{c_{n+1}-2}(123)^{3}721 \\
w_{n+1}(7) &= 1(7)^{q_{n+1}-7}654321
\end{align*}
This order does not fit conditions of Example 1, so we cannot ensure that $\exp(2i\pi/6)$ is a non continuous eigenvalue for every ergodic measure $\mu$ on the system $(X,T)$ induced by this diagram.

As in the previous example one proves that the subdiagrams associated to the subsets of vertices $\set{1,2,3,4,5,6}$ and $\set{7}$ at all levels define systems $(Y,S)$ and $(Z,R)$ respectively, which are uniquely ergodic and the normalized extensions of their unique probability measures, $\widehat{\mu}$ and $\widehat{\nu}$, are ergodic measures on $(X,T)$. Furthermore, a detailed computation allows to prove that the diagram is clean with respect to these measures and that $I_{\widehat{\mu}}=\set{1,2,3,4,5,6}$ and  $I_{\widehat{\nu}}=\set{7}$. This implies there is no other ergodic probability measure on $(X,T)$ aside from such extensions.

Now we prove that $\exp(2i\pi /6)$ is a non continuous eigenvalue for $\widehat{\mu}$ and  that $\widehat{\nu}$ does not have non continuous eigenvalues. The only difference between the model example and this case is the measure of the set
$$
\mathcal{C}_{n}=\{x\in X\talque \overline{s}_{n}(x)\not \equiv_{3} k(\tau_{n}(x),\tau_{n+1}(x))\}.
$$
Here, $\widehat{\mu}(\mathcal{C}_{n}) \leq \frac{2}{q_{n+1}}+ \widehat{\mu}(\tau_{n+1}=7)$ and a simple computation allows to prove that $\sum_{n\geq 1}\widehat{\mu}(\mathcal{C}_{n})$ converges. We deduce by using Corollary \ref{coro:cocycle} that $\exp(2i\pi /6)$ is a non continuous eigenvalue for $\widehat{\mu}$.

The absence of non continuous rational eigenvalues, say $\lambda=\exp(2i\pi/b)$, for $\widehat{\nu}$ follows from inequalities $1< b/(b,p_n) \leq  \#I_{\widehat{\nu}}=1$, which is a contradiction.

\subsection{Example 5: A big variation of the model example} Here we provide a Bratteli-Vershik system of Toeplitz type with rank 7 having two ergodic measures and different non continuous eigenvalues associated to them. The first eigenvalue is $\exp(2i\pi/6)$ and the corresponding $\rb=3$, and the other eigenvalue is $\exp(2i\pi /8)$ with $\rb=4$. In particular, this example shows that all inequalities of Corollary \ref{cor:alternative} (4) can be equalities. 
We keep the values for $q_{n}$ of Example 1 and for $t \in \{1,\ldots, 7\}$ and $n\geq 2$ we consider the following choice of $w_{n+1}(t)$, where $c_n$ is such that $q_n=12c_n+1$:
\begin{align*}
w_{n+1}(1) &= (123)^{4c_{n+1}-2}1245671 \\
w_{n+1}(2) &= 1(312)^{4c_{n+1}-2}345671 \\
w_{n+1}(3) &= 1(123)^{4c_{n+1}-2}145671 \\
w_{n+1}(4) &= 1(5674)^{3c_{n+1}-2}23745671 \\
w_{n+1}(5) &= 15(7456)^{3c_{n+1}-2}7452371 \\
w_{n+1}(6) &= 15(4567)^{3c_{n+1}-2}2367471 \\
w_{n+1}(7) &= 12(5674)^{3c_{n+1}-2}3674571
\end{align*}
As before, one proves that the subdiagrams associated to the sets $\set{1,2,3}$ and $\set{4,5,6,7}$ define systems $(Y,S)$ and $(Z,R)$ respectively which are unique\-ly
ergodic, and the extension of their unique probability measures are ergodic measures on $(X,T)$.
Denote the ergodic measures on $(X,T)$ by $\widehat{\mu}$ and $\widehat{\nu}$. One also has that the diagram is clean and 
$I_{\widehat{\mu}}=\set{1,2,3}$, $I_{\widehat{\nu}}=\set{4,5,6,7}$. Thus, there is no other ergodic probability measure 
on $(X,T)$ aside from $\widehat{\mu}$ and $\widehat{\nu}$.  

Now we sketch a proof that $\lambda=\exp(2i\pi/6)$ is a non continuous eigenvalue for $\widehat{\mu}$. Similarly one proves that $\lambda=\exp(2i\pi/8)$  is a non continuous eigenvalue for $\widehat{\nu}$. This last case is left to the reader.

First, a direct computation (one easily computes nine cases) serves to prove that for any 
$t_{1}, t_{2} \in \{1,2,3\}$, all, up to a bounded number of elements $s \in \bar S_{n}(t_{1},t_{2})$, are constant modulo $3$. 
Denote $k(t_{1},t_{2})$ such a constant. Moreover, if $k(t_{1},1)\equiv_{3} c $ then $k(t_{1},2)\equiv_{3} c+1$ and 
$k(t_{1},3)\equiv_{3}c+2$; and if $k(1,t_{2})\equiv_{3} c'$ then $k(2,t_{2})\equiv_{3}c'+2$ and 
$k(3,t_{2})\equiv_{3}c'+1$. A precise inspection of values of $c$ and $c'$ for all $t_{1}$ and $t_{2}$ allows to prove: 
$$k(t_{1},t_{2})\equiv_{3}k(t_{1},t)+k(t,t_{2}) \text{ for any } t\in \{1,2,3\}.$$
This additive map is the one required by Corollary \ref{coro:cocycle}.
To finish the proof it is enough to produce a set $\mathcal C$ of full measure such that 
for any point  $x \in \mathcal C$ one has 
$\bar s_{n}(x)\equiv_{3}k(\tau_{n}(x),\tau_{n+1}(x))$ for all enough large $n\in \NN$.
As before, by considering for any $n\in \NN$ the set
$$\mathcal C_{n}=\{x \in X; \bar s_{n}(x)\not\equiv_{3}k(\tau_{n}(x),\tau_{n+1}(x))\}$$
and using the fact that any $s \in \bar S_{n}(t_{1},t_{2})$ up to a bounded number of elements, say $L$,  is constant modulo $3$, one gets that $\hat \mu(\mathcal{C}_{n})\leq 3L/q_{n+1}$. We finish the proof of the claim 
by the Borel-Cantelli Lemma, taking 
$\displaystyle\mathcal{C}=\liminf_{n\to \infty} \mathcal{C}_n^{c}$.

\subsection{Example 6: Another (similar) big variation of the model example} Here we modify the previous example 
to provide a system with two ergodic measures and non continuous eigenvalues $\exp(2i\pi/6)$ and $\exp(2i\pi/4)$ respectively. This example shows that the first inequality of Corollary \ref{cor:alternative} (4) is an equality 
and the second is a strict inequality. For $t\in \{1,\ldots,7\}$ and $n\geq 2$, consider the following choice of 
$w_{n+1}(t)$ and write $q_n=12c_n+1$: 
\begin{align*}
w_{n+1}(1) &= (123)^{4c_{n+1}-2}1245671 \\
w_{n+1}(2) &= 1(312)^{4c_{n+1}-2}345671 \\
w_{n+1}(3) &= 1(123)^{4c_{n+1}-2}145671 \\
w_{n+1}(4) &= 1(647465)^{2c_{n+1}-1}237461 \\
w_{n+1}(5) &= 1(656574)^{2c_{n+1}-1}652361 \\
w_{n+1}(6) &= 16(646575)^{2c_{n+1}-1}72361 \\
w_{n+1}(7) &= 16(757564)^{2c_{n+1}-1}73261
\end{align*}
In this example the subdiagrams associated to $\set{1,2,3}$ and $\set{4,5,6,7}$ define systems $(Y,S)$ and $(Z,R)$ respectively which are uniquely ergodic, and the extensions of these ergodic measures, $\widehat{\mu}$ and $\widehat{\nu}$, are  ergodic probability measures in $(X,T)$. As in the previous example there is no other ergodic probability measure on $(X,T)$. Furthermore, the diagram is clean,  $I_{\widehat{\mu}}=\set{1,2,3}$ and  $I_{\widehat{\nu}}=\set{4,5,6,7}$.

In relation to eigenvalues, doing similar computations as in the previous example one gets that 
$\exp(2\pi i/6)$ is a non continuous eigenvalue for  $\widehat{\mu}$ and that $\exp(2\pi i/4)$ is a non continuous 
eigenvalue for $\widehat{\nu}$, while $\exp(2\pi i/8)$ is not.

{\small
\bibliography{biblio}{}
\bibliographystyle{amsalpha}
}
\end{document}